\documentclass[a4paper,11pt]{article}
\usepackage[latin1]{inputenc}
\usepackage[english]{babel}
\usepackage{amsmath}
\usepackage{amsfonts}
\usepackage{amssymb}
\usepackage{epsfig}
\usepackage{amsopn}
\usepackage{amsthm}
\usepackage{color}
\usepackage{graphicx}
\usepackage{subfigure}
\usepackage{enumerate}
\newtheorem{theorem}{Theorem}[section]
\newtheorem{corollary}[theorem]{Corollary}
\newtheorem{lemma}[theorem]{Lemma}
\newtheorem{proposition}[theorem]{Proposition}

\newtheorem*{theorem*}{Theorem}
\newtheorem*{lemma*}{Lemma}
\newtheorem*{remark*}{Remark}
\newtheorem*{definition*}{Definition}
\newtheorem*{proposition*}{Proposition}
\newtheorem*{corollary*}{Corollary}
\numberwithin{equation}{section}
\newcommand{\real}{\mathbb{R}}

\let\ced=\c         

\def\qed{\,\unskip\kern 6pt \penalty 500
\raise -2pt\hbox{\vrule \vbox to8pt{\hrule width 6pt
\vfill\hrule}\vrule}\par}
\definecolor{darkblue}{rgb}{0.05, .05, .65}
\definecolor{darkgreen}{rgb}{0.1, .65, .1}
\definecolor{darkred}{rgb}{0.8,0,0}
\newcommand{\beqn}{\begin{equation}}
\newcommand{\eeqn}{\end{equation}}
\newcommand{\bear}{\begin{eqnarray}}
\newcommand{\eear}{\end{eqnarray}}
\newcommand{\bean}{\begin{eqnarray*}}
\newcommand{\eean}{\end{eqnarray*}}
\begin{document}

\title{\huge \bf Self-similar solutions preventing finite time blow-up for reaction-diffusion equations with singular potential}

\author{
\Large Razvan Gabriel Iagar\,\footnote{Departamento de Matem\'{a}tica
Aplicada, Ciencia e Ingenieria de los Materiales y Tecnologia
Electr\'onica, Universidad Rey Juan Carlos, M\'{o}stoles,
28933, Madrid, Spain, \textit{e-mail:} razvan.iagar@urjc.es},\\
[4pt] \Large Ana Isabel Mu\~{n}oz\,\footnote{Departamento de Matem\'{a}tica
Aplicada, Ciencia e Ingenieria de los Materiales y Tecnologia
Electr\'onica, Universidad Rey Juan Carlos, M\'{o}stoles,
28933, Madrid, Spain, \textit{e-mail:} anaisabel.munoz@urjc.es},
\\[4pt] \Large Ariel S\'{a}nchez,\footnote{Departamento de Matem\'{a}tica
Aplicada, Ciencia e Ingenieria de los Materiales y Tecnologia
Electr\'onica, Universidad Rey Juan Carlos, M\'{o}stoles,
28933, Madrid, Spain, \textit{e-mail:} ariel.sanchez@urjc.es}\\
[4pt] }
\date{}
\maketitle

\begin{abstract}
We prove existence and uniqueness of a global in time self-similar solution growing up as $t\to\infty$ for the following reaction-diffusion equation with a singular potential
$$
u_t=\Delta u^m+|x|^{\sigma}u^p,
$$
posed in dimension $N\geq2$, with $m>1$, $\sigma\in(-2,0)$ and $1<p<1-\sigma(m-1)/2$. For the special case of dimension $N=1$, the same holds true for $\sigma\in(-1,0)$ and similar ranges for $m$ and $p$. The existence of this global solution prevents finite time blow-up even with $m>1$ and $p>1$, showing an interesting effect induced by the singular potential $|x|^{\sigma}$. This result is also applied to reaction-diffusion equations with general potentials $V(x)$ to prevent finite time blow-up via comparison.
\end{abstract}

\


\smallskip


\section{Introduction}

The goal of this paper is to establish existence and uniqueness of a self-similar solution presenting grow-up as $t\to\infty$ for the following reaction-diffusion equation with a singular potential in the reaction term
\begin{equation}\label{eq1}
u_t=\Delta u^m+|x|^{\sigma}u^p,
\end{equation}
posed for $(x,t)\in\real^N\times(0,\infty)$ with dimension $N\geq1$. Throughout this work, the exponents in Eq. \eqref{eq1} satisfy
\begin{equation}\label{exp.range}
m>1, \qquad -2<\sigma<0, \qquad 1\leq p<p_c(\sigma):=1-\frac{\sigma(m-1)}{2},
\end{equation}
noticing that the critical exponent $p_c(\sigma)\in(1,m)$ for $\sigma\in(-2,0)$. We are thus dealing with a singular potential of Hardy-type acting on the reaction term, and our main goal is to put into evidence the interesting (and sometimes unexpected) influence of this potential on the dynamics of the equation.

The competition between diffusion (either linear or quasilinear) and terms with singular potentials became a fashionable research subject in the last decades, starting from the well known paper by Baras and Goldstein \cite{BG84}. In that paper, the Hardy potential $K/|x|^2$ is considered and it is shown that existence or non-existence of solutions (in suitable functional spaces) to the linear reaction-diffusion equation
\begin{equation}\label{eq.linear}
u_t=\Delta u+\frac{K}{|x|^2}u
\end{equation}
is strongly related to the Hardy inequality with optimal constant $(N-2)^2/4$. Indeed, for constants $K>(N-2)^2/4$ no solutions exist as they blow-up instantaneously, while for the opposite case or for potentials $K/|x|^{\alpha}$ with $0<\alpha<2$ existence of solutions is established. Later on, Cabr\'e and Martel \cite{CM99} extended the threshold between existence and non-existence for Eq. \eqref{eq.linear} to general potentials $a(x)\in L^1_{\rm loc}(\Omega)$ instead of $|x|^{-2}$ (where $\Omega\subset\real^N$ is a bounded domain), establishing a sharp condition for the non-existence of solutions in terms of the spectrum of the (linear) operator $-\Delta-a(x)$. The connections between the Hardy inequality and the properties of solutions to the heat equation with a Hardy potential were later developed in \cite{VZ00} employing more refined functional inequalities such as the sharp Hardy-Poincar\'e inequality (see also \cite{VZ12}). Using similar techniques as in \cite{CM99}, Goldstein and Zhang \cite{GZ02} extend the study of existence and non-existence to general linear parabolic operators with variable coefficients in the leading order term.

The mathematical study of the semilinear problem with singular potential, also known as the Hardy equation
\begin{equation}\label{eq.semi}
u_t=\Delta u+K|x|^{\sigma}u^p, \qquad p>1, \ -2<\sigma<0,
\end{equation}
saw a number of very recent papers addressing the question of existence of solutions in suitable functional spaces or for the Cauchy problem with weakly regular initial conditions (see works such as \cite{BSTW17, BS19, CIT21a, CIT21b, T20}) or even singular ones \cite{HT21}. The same question of existence of global solutions is studied with fractional diffusion \cite{HS21, HT21}. Going into finer properties related to the dynamics of the equation Eq. \eqref{eq.semi}, such as similarity solutions and behavior near blow-up for general solutions, Filippas and Tertikas \cite{FT00} proved that for $-2<\sigma<0$ there exists a unique, decreasing blow-up self-similar solution to Eq. \eqref{eq.semi} provided that $1<p<p_S=(N+2+2\sigma)/(N-2)$ and then an infinity of such self-similar solutions for $p_S<p<p_{JL}$, where $p_{JL}$ is a larger critical exponent known as the Joseph-Lundgren exponent (whose complicated expression we omit here). Their analysis has been completed very recently by Mukai and Seki \cite{MS21} with the range $p>p_{JL}$, for which exact blow-up rates and asymptotic expansions in regions are established.

The competition between quasilinear diffusion and singular potentials of Hardy type has been also investigated. Some of the first works (up to our knowledge) dealing with Eq. \eqref{eq1} are \cite{Qi98, Su02}, where the Fujita-type exponent $p_F(\sigma)=m+(2+\sigma)/N$ is established, the fast diffusion case $(N-2)/N<m<1$ also being considered. However, these authors only work with $p>m$, which is authomatic if $m<1$ but leaves aside a full range of parameters for $m>1$. Similar to Fujita's seminal work \cite{Fu66}, it is proved in \cite{Qi98} that, for $\max\{1,m\}<p\leq p_F(\sigma)$ all the non-trivial solutions blow up in finite time, while for $p>p_F(\sigma)$ global solutions in self-similar form are constructed and the analysis for $p>p_F(\sigma)$ is then refined in \cite{Su02}. Apart from the explicit value of the Fujita-type exponent depending on $\sigma$, it appears that in these ranges of exponents the qualitative behavior is similar to the one without singular potential (that is, with $\sigma=0$, see \cite{S4}). Later on, different quasilinear diffusion operators were considered together with the Hardy-type potential: fast diffusion $0<m<1$ in \cite{GK03, GGK05}, $p$-Laplacian diffusion in \cite{AP00} and doubly nonlinear diffusion in \cite{Ko04}. Very recently, the authors considered the borderline case $p=m$ with the exact Hardy potential $K|x|^{-2}$ in a short note \cite{IS20b} and established, by means of a transformation to the porous medium equation, an interesting case of continuation after blow-up: all the solutions blow up (either in finite time or even instantaneously) only at $x=0$ but they keep belonging to the suitable weighted spaces allowing for the development of the weak theory for any $t>0$. Once more, the optimal Hardy constant $K=(N-2)^2/4$ comes into play limiting the existence and non-existence ranges also in this specific quasilinear case treated in \cite{IS20b}.

Having these precedents as starting points, we describe below our main results and some quite surprising applications of them.

\medskip

\noindent \textbf{Main results.} As we already have said, the main goal of this paper is to construct a unique, global in time self-similar solution to Eq. \eqref{eq1} in the range of exponents \eqref{exp.range}. More specifically, we look for solutions with radial symmetry in the form
\begin{equation}\label{SSS}
u(x,t)=t^{\alpha}f(\xi), \qquad \xi=|x|t^{-\beta},
\end{equation}
with $\alpha>0$, $\beta>0$ exponents to be found. Introducing the ansatz \eqref{SSS} into Eq. \eqref{eq1}, we find that the self-similarity exponents are
\begin{equation}\label{self.exp}
\alpha=-\frac{\sigma+2}{\sigma(m-1)+2(p-1)}>0, \qquad \beta=-\frac{m-p}{\sigma(m-1)+2(p-1)}>0,
\end{equation}
which are positive precisely because, in the range
$$
1\leq p<p_c(\sigma)=1-\frac{\sigma(m-1)}{2}, \qquad p_c(\sigma)\in(1,m),
$$
we have $\sigma(m-1)+2(p-1)<0$. The similarity profiles $f(\xi)$ solve the non-autonomous differential equation
\begin{equation}\label{ODE}
(f^m)''(\xi)+\frac{N-1}{\xi}(f^m)'(\xi)-\alpha f(\xi)+\beta\xi f'(\xi)+\xi^{\sigma}f(\xi)^p=0.
\end{equation}
Let us stress here that the range of exponents \eqref{exp.range} is not covered in the above mentioned work \cite{Qi98}, which deals with the case $p>m$, and up to our knowledge, it remained open. The main theorem of the present work is the following:
\begin{theorem}\label{th.1}
Let $m$, $p$ and $\sigma$ as in \eqref{exp.range} if $N\geq2$ or with the extra restriction $\sigma>-1$ if $N=1$. There exists a \emph{unique profile} $f(\xi)$ solution to \eqref{ODE} with the following \emph{local behavior at the origin}
\begin{equation}\label{beh.Q1}
f(\xi)\sim\left[D(\sigma)-\frac{m-p}{m(N+\sigma)(\sigma+2)}\xi^{\sigma+2}\right]^{1/(m-p)}, \qquad {\rm as} \ \xi\to0,
\end{equation}
with $D(\sigma)>0$, and having an \emph{interface} at some finite point $\xi_0\in(0,\infty)$ in the sense that
\begin{equation}\label{beh.interf}
f(\xi)>0 \ {\rm for} \ \xi\in(0,\xi_0), \qquad f(\xi_0)=0, \ (f^m)'(\xi_0)=0.
\end{equation}
Moreover, the profile $f(\xi)$ is decreasing on $(0,\xi_0)$.
\end{theorem}
The condition $(f^m)'(\xi_0)=0$ in \eqref{beh.interf} is the standard contact condition at the interface point ensuring that, in a neighborhood of this point, the function defined as in \eqref{SSS} starting from the profile $f(\xi)$ becomes a weak solution for the porous medium diffusion, as explained in \cite[Section 9.8]{VPME}. The restriction $\sigma>-1$ in dimension $N=1$ is a common feature in the theory of Hardy-type equations (see for example \cite{BSTW17, HT21}).

Let us also remark at this point that, in the special range $N\geq2$ and $-2<\sigma\leq-1$, the profiles $f(\xi)$ defined in \eqref{beh.Q1} fail from having a good slope at $\xi=0$. This is a fact that occurs when dealing with singular coefficients; just as an example a similar situation has been also found in \cite{RV06} in the study of the non-homogeneous porous medium equation (where the singular potential, or density, appears in front of $u_t$), where such profiles with a peak at $\xi=0$ are found and it is shown that they are fundamental for the dynamics of the equation. We will keep calling them solutions, allowing ourselves for now an "abuse of language", as they have totally similar properties as in the range $\sigma>-1$ and will be true solutions in suitable (weighted) spaces in which the general theory of Eq. \eqref{eq1} is developed, also being the candidates for the large time behavior of general solutions in this range. This formation of a "local peak" at $x=0$ (blow-up of the derivative of the profile) for $\sigma<-1$ is intuitively expected, since the behavior at $x=0$ should converge to a vertical asymptote as $\sigma\to-2$ (as shown in \cite{IS20b}).  We refrain here from entering the functional-analytic questions related to the general notion of solution to Eq. \eqref{eq1}, as all these questions (including existence, uniqueness or non-uniqueness of general solutions, and large time behavior) will be addressed in the forthcoming paper \cite{IMSPrep} where the self-similar solutions in Theorem \ref{th.1} are strongly used. Instead of this, we give below some interesting applications that can be deduced directly from the self-similar solutions. We picture in Figure \ref{fig1} some profiles starting with the behavior \eqref{beh.Q1} at $\xi=0$, in both cases $-1<\sigma<0$ and $-2<\sigma<-1$ to illustrate the local difference among them.

\begin{figure}[ht!]
  \begin{center}
  \subfigure[$\sigma\in(-1,0)$]{\includegraphics[width=7.5cm,height=6cm]{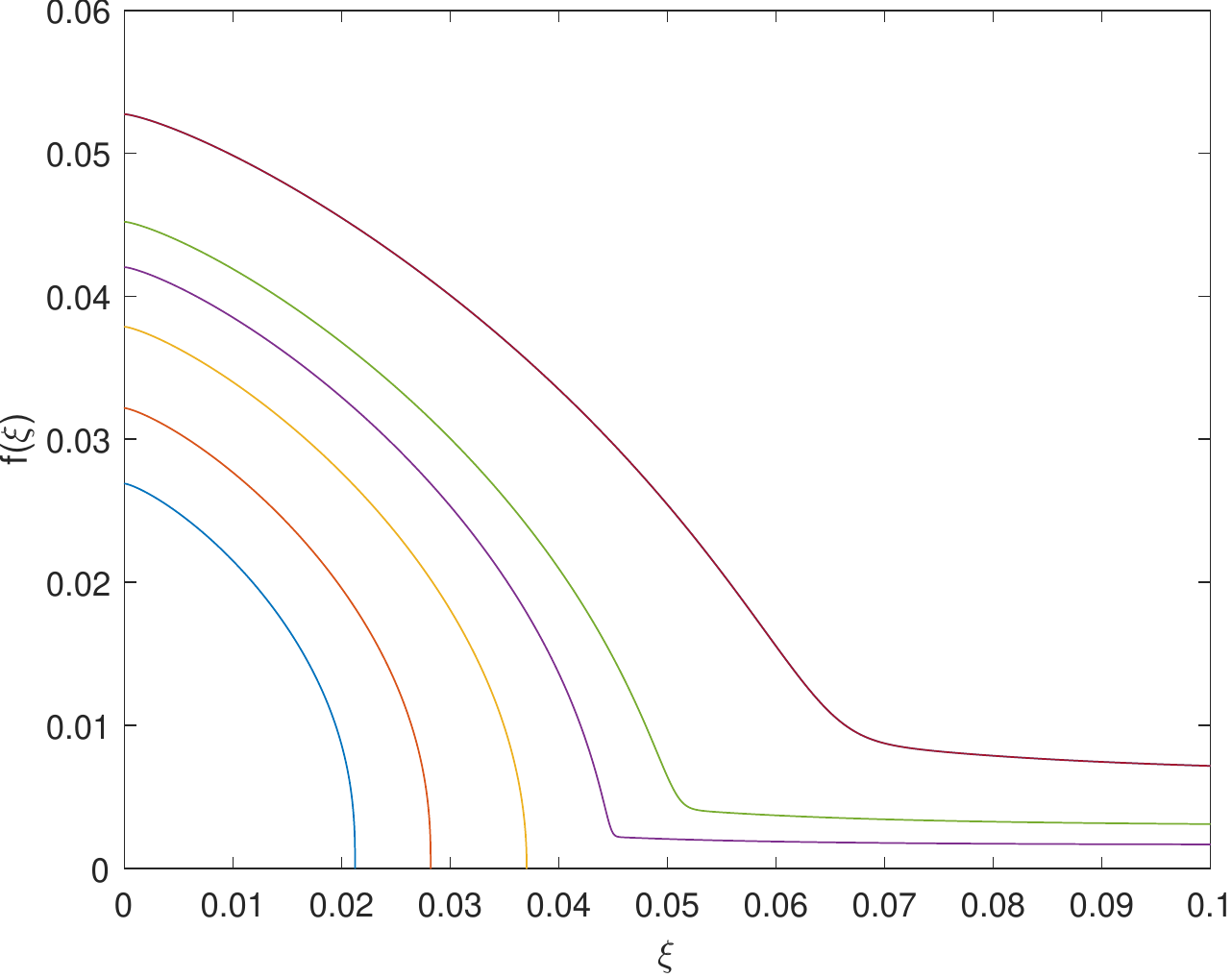}}
  \subfigure[$\sigma\in(-2,-1)$]{\includegraphics[width=7.5cm,height=6cm]{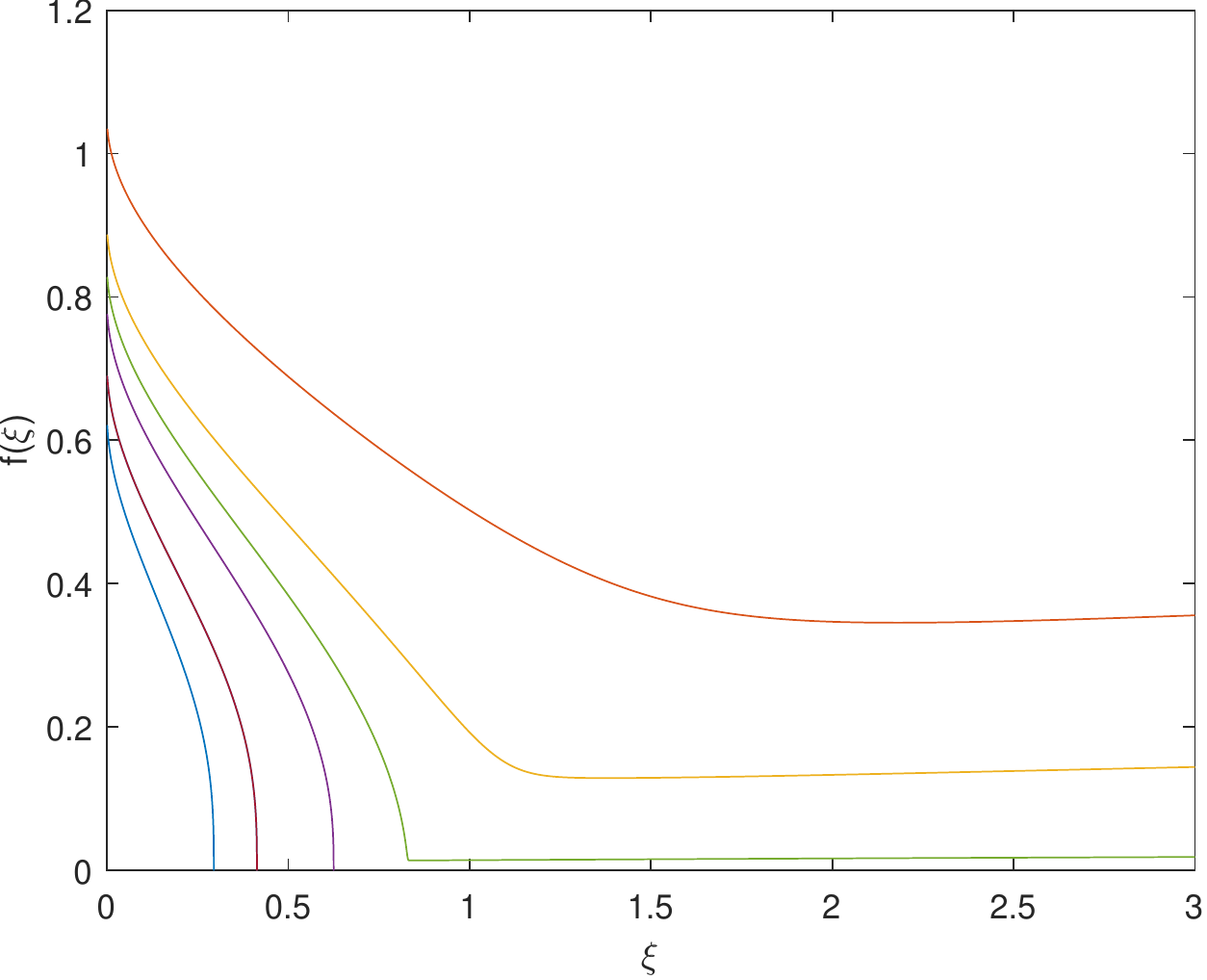}}
  \end{center}
  \caption{A shooting of profiles with local behavior \eqref{beh.Q1}. Experiments for $m=3$, $N=3$, $p=1.2$ and $\sigma=-0.7$, respectively $p=1.5$ and $\sigma=-1.5$}\label{fig1}
\end{figure}

\medskip

\noindent \textbf{Applications of the main results.} We gather below several nice consequences of the existence of profiles as in Theorem \ref{th.1}.

\medskip

\noindent \textbf{1. No finite time blow-up}. Since the solution constructed in Theorem \ref{th.1} is global, we already have an example of solution without finite time blow-up to Eq. \eqref{eq1}. This is one of the most striking effects of the presence of the singular potential: indeed, it is well known that, if $\sigma=0$, \emph{any solution} to Eq. \eqref{eq1} with $m>1$ and $1<p<m$ blows up in finite time, as shown for example in the monograph \cite{S4}. While in our case, a \emph{strong contrast appears}, as the singular potential acts in the sense of preventing blow-up. An \emph{intuitive explanation} of this phenomenon would be the following: in the competition between the diffusion (spreading the mass) and the reaction term, if $p>1$ is not sufficiently big, mass spreads faster far away from the neighborhood of $x=0$ where the singular potential is very strong, thus entering the region where the potential is weaker and limits the effect of the reaction term. More generally, provided some comparison principle is allowed, since the solution \eqref{SSS} expands at the same time in amplitude and support, it would stop \emph{any other solution} to Eq. \eqref{eq1} from blowing up in finite time. This question of comparison is not easy and will also be considered in the paper in preparation \cite{IMSPrep}.

\medskip

\noindent \textbf{2. Preventing blow-up for other equations with potentials}. The solution \eqref{SSS} with the profile $f(\xi)$ given by Theorem \ref{th.1} gives an upper bound to other equations with weighted reaction terms. Indeed, if we let $V(|x|)\in C(\real)$ to be a radially symmetric, continuous function such that
$$
V(0)<\infty, \qquad V(|x|)\leq |x|^{\sigma}, \ {\rm for \ any} \ x\in\real^N\setminus\{0\},
$$
then our self-similar solution becomes a supersolution to the equation
\begin{equation}\label{potV}
u_t=\Delta u^m+V(|x|)u^p, \qquad m>1, \ 1\leq p<p_c(\sigma),
\end{equation}
which has a comparison principle according to standard theory since $V(|x|)$ is bounded and $p\geq1$. Thus, taking into account that our self-similar solution \eqref{SSS} expands simultaneously in amplitude and support as $t>0$ increases, it can be put above any compactly supported and bounded initial condition, if given sufficient time. Hence, comparison shows that \emph{any solution} to Eq. \eqref{potV} is global in time. A noticeable case is the family of equations
\begin{equation}\label{eq.AdB}
u_t=\Delta u^m+(1+|x|)^{\sigma}u^p, \qquad \sigma<0,
\end{equation}
analyzed in great detail by Andreucci and DiBenedetto in \cite{AdB91}. We can thus derive that all the solutions to Eq. \eqref{eq.AdB} with compactly supported initial conditions are global in time. We also recall here that compactly supported potentials $V(x)$ have been considered in \cite{FdPV06, FdP18}, where it is shown that for dimension $N=1$, the global existence exponent is $p_1=(m+1)/2$, while for $N\geq2$, the global existence exponent is $p_2=m$. We observe that these exponents coincide with $p_c(\sigma)$ taken exactly in our limit cases $\sigma=-1$ for $N=1$ and $\sigma=-2$ for $N\geq2$.

\medskip

\noindent \textbf{3. Non-uniqueness for Eq. \eqref{eq1} with zero initial trace}. Let us notice that the solution \eqref{SSS} introduced in Theorem \ref{th.1} has zero initial trace. Thus, from the trivial initial condition $u_0\equiv0$ we obtain two different solutions. This kind of non-uniqueness has been noticed in ranges of exponents where forward self-similar solutions exist for semilinear reaction-diffusion equations (that is, with $m=1$), for example for some reaction exponents $p$ larger than the Fujita exponent, see \cite{HW82} and \cite[Remark (ii), page 77]{QS}. However, Haraux and Weissler \cite{HW82} prove non-uniqueness with respect to data having zero initial trace only in suitable $L^p$ spaces, since the self-similar solutions in \cite{HW82} concentrate at $x=0$ with infinite pointwise limit, while our self-similar solution takes zero initial trace in $L^{\infty}$ sense. This opens up an interesting question in the same line as the one addressed in \cite{W85}: for which initial conditions $u_0$, uniqueness holds true for Eq. \eqref{eq1}? We will try to give an answer to it in the forthcoming paper \cite{IMSPrep}.

\medskip

As a final comment, our analysis includes the lower limit case $p=1$ but not the upper limit case $p=p_c(\sigma)$, where a different type of self-similar solutions, in exponential form and with "eternal" life span, is expected, since $\sigma(m-1)+2(p-1)=0$, in line with papers such as \cite{IS21c}.

\medskip

\noindent \textbf{Some comments on the proofs and organization of the paper.} The techniques we use in the proof of Theorem \ref{th.1} are very different for the existence and for the uniqueness part. Existence will follow from a detailed analysis of an autonomous, quadratic dynamical system in line with the study performed by two of the authors in their previous papers \cite{IS19, IS21a} dealing with the range $1\leq p<m$ but with $\sigma>0$. A number of differences will appear in the analysis, mostly not due to the change of sign of $\sigma$ but to the change of sign of the very influential constant $\sigma(m-1)+2(p-1)$. Uniqueness is proved by using a \emph{sliding} technique inspired from \cite{IV10, YeYin} based on the construction of optimal barriers for the solution via rescaling. Due to technical reasons, dimension $N=1$ comes with some differences, thus, for the readers' convenience, the main flow of the proofs will be given in dimension $N\geq2$, while dimension $N=1$ will be considered separately at the end of the paper. Thus, Section \ref{sec.local} deals with the dynamical system and the analysis of its finite critical points. It is then followed by Section \ref{sec.infty} where the critical points at infinity of the dynamical system are analyzed, thus completing the local analysis of the system. Existence of the self-similar profile will follow from a shooting technique performed in the phase-space, in Section \ref{sec.exist}. The monotonicity of all the profiles with behavior as in \eqref{beh.Q1} and the uniqueness of the profile with interface will be the subject of Section \ref{sec.uniq}, thus completing the proof of Theorem \ref{th.1} in dimension $N\geq2$. The paper ends with a section devoted to the special dimension $N=1$, explaining the local differences at the points where they appear and, in particular, why in dimension $N=1$ we have to modify the statement and only consider the range $\sigma\in(-1,0)$.

\section{The phase space. Local analysis}\label{sec.local}

As explained in the Introduction, our main object of study is the non-autonomous differential equation \eqref{ODE}. In order to study the possible behaviors of its solutions, we transform it into an autonomous dynamical system by the following change of variable, adapting the one used in previous papers \cite{IS19, IS21a}
\begin{equation}\label{PSchange}
X(\eta)=\frac{m}{\alpha}\xi^{-2}f(\xi)^{m-1}, \ \ Y(\eta)=\frac{m}{\alpha}\xi^{-1}f(\xi)^{m-2}f'(\xi), \ \ Z(\eta)=\frac{1}{\alpha}\xi^{\sigma}f(\xi)^{p-1},
\end{equation}
where the new independent variable $\eta$ is introduced via the differential equation
\begin{equation}\label{ind.var}
\frac{d\eta}{d\xi}=\frac{\alpha}{m}\xi f(\xi)^{1-m}.
\end{equation}
By expressing $f'(\xi)$ in terms of $Y$ from the second equation in \eqref{PSchange} and replacing it into \eqref{ODE}, we find after straightforward calculations the following three-dimensional dynamical system in which Eq. \eqref{ODE} is transformed (where the dot derivatives are understood with respect to the new variable $\eta$):
\begin{equation}\label{PSsyst1}
\left\{\begin{array}{ll}\dot{X}=X[(m-1)Y-2X],\\
\dot{Y}=-Y^2-\frac{\beta}{\alpha}Y+X-NXY-XZ,\\
\dot{Z}=Z[(p-1)Y+\sigma X].\end{array}\right.
\end{equation}
Notice that in the system \eqref{PSsyst1} the planes $\{X=0\}$ and $\{Z=0\}$ are invariant for the system and that $X\geq0$, $Z\geq0$, only $Y$ being allowed to change sign. The system generalizes in dimension $N$ the one used in \cite{IS21a}, but the main differences with respect to its study will be the effect of the negativity of the constant $\sigma(m-1)+2(p-1)$, a factor that introduces some significant changes with respect to the analysis performed in the above quoted work. Assume now that $p>1$. The critical points in the finite part of the phase space associated to the system \eqref{PSsyst1} are
$$
P_0=(0,0,0), \ \ P_1=\left(0,-\frac{\beta}{\alpha},0\right), \ \ {\rm and} \ P^{\gamma}=(0,0,\gamma),
$$
for any $\gamma>0$. We give below the local analysis of the flow of the system in a neighborhood of each critical point.

\begin{lemma}[Local analysis near $P_0$]\label{lem.P0}
The critical point $P_0$ behaves like an attractor for the trajectories of the system \eqref{PSsyst1}. The orbits entering it contain profiles with the local behavior
\begin{equation}\label{beh.P0}
f(\xi)\sim K\xi^{(\sigma+2)/(m-p)}, \qquad {\rm as} \ \xi\to\infty, \qquad K>0 \ {\rm free \ constant},
\end{equation}
thus increasing to infinity.
\end{lemma}
\begin{proof}
The linearization of the system \eqref{PSsyst1} near $P_0$ has the matrix
$$
M(P_0)=\left(
         \begin{array}{ccc}
           0 & 0 & 0 \\
           1 & -\frac{\beta}{\alpha} & 0 \\
           0 & 0 & 0 \\
         \end{array}
       \right),
$$
with a negative eigenvalue and a two-dimensional center manifold. We next study the flow of the system on the center manifold by using the Center Manifold Theorem \cite[Theorem 1, Section 2.12]{Pe}. To this end, we first replace the variable $Y$ by $W=(\beta/\alpha)Y-X$ in order to get the canonical form of the system \eqref{PSsyst1} near $P_0$:
\begin{equation}\label{interm1}
\left\{\begin{array}{ll}\dot{X}=-\frac{1}{\beta}X^2+\frac{(m-1)\alpha}{\beta}XW,\\
\dot{W}=-\frac{\beta}{\alpha}W-\frac{\alpha}{\beta}W^2-\frac{\alpha(m+1)+N\beta}{\beta}XW-\frac{\beta}{\alpha}XZ-\frac{(N-2)\beta+m\alpha}{\beta}X^2,\\
\dot{Z}=-\frac{1}{\beta}XZ+\frac{\alpha(p-1)}{\beta}ZW.\end{array}\right.
\end{equation}
We now look for a center manifold of the form
$$
W=h(X,Z)=aX^2+bXZ+cZ^2+O(|(X,Z)|^3),
$$
with $a$, $b$ and $c$ coefficients to be determined. By replacing this form into the equation of the center manifold given in \cite[Theorem 1, Section 2.12]{Pe} and identifying only the quadratic terms, we readily obtain the expansion of the center manifold near $P_0$ as
$$
h(X,Z)=-\frac{\alpha}{\beta^2}[(N-2)\beta+\alpha m]X^2-XZ+XO(|(X,Z)|^2),
$$
while the flow on the center manifold is given by the reduced system
\begin{equation}\label{interm2}
\left\{\begin{array}{ll}\dot{X}&=-\frac{1}{\beta}X^2+X^2O(|(X,Z)|),\\
\dot{Z}&=-\frac{1}{\beta}XZ+XO(|(X,Z)|^2),\end{array}\right.
\end{equation}
in a neighborhood of its origin $(X,Z)=(0,0)$. It thus follows that the orbits on the center manifold enter the point $P_0$. The system \eqref{interm2} can be integrated in first approximation to get $Z\sim KX$ for any free constant $K>0$, deriving thus the behavior \eqref{beh.P0} by replacing $Z$ and $X$ in terms of profiles with their definitions in \eqref{PSchange}. Since $X\to0$ at $P_0$, we get that
$$
\xi^{-2}f(\xi)^{m-1}\sim\xi^{-2+(m-1)(\sigma+2)/(m-p)}=\xi^{[\sigma(m-1)+2(p-1)]/(m-p)}\to0,
$$
thus the behavior \eqref{beh.P0} is taken as $\xi\to\infty$ owing to the fact that $\sigma(m-1)+2(p-1)<0$ in our range \eqref{exp.range}, ending the proof.
\end{proof}
\begin{lemma}[Local analysis near $P_1$]\label{lem.P1}
The system \eqref{PSsyst1} in a neighborhood of $P_1$ has a two-dimensional stable manifold and a one-dimensional unstable manifold. The orbits entering $P_1$ on the two-dimensional manifold correspond to the interface behavior \eqref{beh.interf} at a point $\xi_0\in(0,\infty)$, with the more precise local behavior
\begin{equation}\label{beh.P1}
f(\xi)\sim\left[C-\frac{\beta(m-1)}{2m}\xi^2\right]_{+}^{1/(m-1)}, \qquad {\rm as} \ \xi\to\xi_0=\sqrt{\frac{2mC}{\beta(m-1)}}, \ \xi<\xi_0,
\end{equation}
where $C>0$ is a free constant.
\end{lemma}
\begin{proof}
The linearization of the system in a neighborhood of $P_1$ has the matrix
$$
M(P_1)=\left(
         \begin{array}{ccc}
           -\frac{(m-1)\beta}{\alpha} & 0 & 0 \\
           1+\frac{N\beta}{\alpha} & \frac{\beta}{\alpha} & 0 \\
           0 & 0 & -\frac{(p-1)\beta}{\alpha} \\
         \end{array}
       \right)
$$
with eigenvalues $\lambda_1=-(m-1)\beta/\alpha$, $\lambda_2=\beta/\alpha$ and $\lambda_3=-(p-1)\beta/\alpha$, and corresponding eigenvectors
$$
e_1=\left(1,-\frac{N\beta+\alpha}{m\beta},0\right), \ e_2=(0,1,0), \ e_3=(0,0,1).
$$
We thus obtain a two-dimensional stable manifold formed by trajectories entering $P_1$ tangent to the plane spanned by $e_1$ and $e_3$, while the one-dimensional manifold is contained in the $Y$ axis. The orbits entering $P_1$ on the stable manifold contain profiles such that
$$
Y=\frac{m}{\alpha}\xi^{-1}f(\xi)^{m-2}f'(\xi)\sim-\frac{\beta}{\alpha},
$$
whence by integration we find the local behavior \eqref{beh.P1}. We readily notice that \eqref{beh.P1} is more precise than \eqref{beh.interf} as it fulfills the interface condition $(f^m)'(\xi_0)=0$.
\end{proof}
The critical points $P^{\gamma}$ do not introduce any interesting orbit, as it is proved in the following rather tedious part.
\begin{lemma}\label{lem.Pgamma}
There are no interesting profiles $f(\xi)$ contained in orbits of the system \eqref{PSsyst1} connecting to or from any of the points $P^{\gamma}$ with $\gamma>0$.
\end{lemma}
\begin{proof}
The local analysis near a point $P^{\gamma}$ with some $\gamma>0$ follows the same steps as in the study performed in \cite[Lemma 2.4]{IS21a} but with a different outcome. In a first step, we have to translate the critical point to the origin by letting $Z=\overline{Z}+\gamma$ and obtaining a new system
\begin{equation}\label{interm3}
\left\{\begin{array}{ll}\dot{X}=X[(m-1)Y-2X],\\
\dot{Y}=-Y^2-\frac{\beta}{\alpha}Y+(1-\gamma)X-NXY-X\overline{Z},\\
\dot{\overline{Z}}=\overline{Z}[(p-1)Y+\sigma X]+\sigma\gamma X+(p-1)\gamma
Y.\end{array}\right.
\end{equation}
We readily notice that the linearization of the system \eqref{interm3} near the origin has a one-dimensional stable manifold and a two-dimensional center manifold. The most involved part of the analysis is the study of the center manifold, following the strategy in \cite[Section 2.12]{Pe}. To this end, we perform a double change of variable in \eqref{interm3} in order to reduce some linear terms and to put it into a canonical form, by replacing $Y$ and $\overline{Z}$ with the new variables $G$ and $H$ defined by
$$
G=\frac{\beta}{\alpha}Y-(1-\gamma)X, \qquad H=\overline{Z}+kY, \qquad k=\frac{\alpha(p-1)\gamma}{\beta},
$$
to get the following new system in variables $(X,G,H)$:
\begin{equation}\label{interm4}
\left\{\begin{array}{ll}\dot{X}=\left[\frac{(m-1)\alpha(1-\gamma)}{\beta}-2\right]X^2+\frac{(m-1)\alpha}{\beta}GX,\\
\dot{G}=-\frac{\beta}{\alpha}G-\frac{\alpha}{\beta}G^2-\frac{\beta}{\alpha}HX-D_1X^2-D_2GX,\\
\dot{H}=(\gamma\sigma+k(1-\gamma))X+D_3XH-D_4X^2-\frac{\alpha^2kp}{\beta^2}G^2-D_5GX+\frac{\alpha(p-1)}{\beta}GH,\end{array}\right.
\end{equation}
with coefficients
\begin{equation}\label{interm4b}
\begin{split}
&D_1=\frac{(1-\gamma)[(N-2)\beta+m\alpha(1-\gamma)-\alpha\gamma(p-1)]}{\beta}, \ D_2=\frac{N\beta+\alpha(m+1)-\alpha\gamma(m+p)}{\beta},\\
&D_3=\sigma-k+\frac{(p-1)\alpha(1-\gamma)}{\beta}, \ D_4=\frac{(1-\gamma)\alpha^2\gamma(p-1)[(N+\sigma)\beta+\alpha(\gamma+p-2p\gamma)]}{\beta^3},\\
&D_5=\frac{\alpha^2\gamma(p-1)[(N+\sigma)\beta+\alpha(\gamma+2p-3p\gamma)]}{\beta^3}.
\end{split}
\end{equation}
By applying the center manifold theorem \cite[Theorem 1, Section 2.12]{Pe} to the new system \eqref{interm4} and employing an argument by induction, we obtain that the two-dimensional center manifold at the origin in \eqref{interm4} has the form
$$
G=AX^2-XH+XO(|(X,H)|^2), \qquad A=\frac{\alpha}{\beta}\left[\gamma\sigma+k(1-\gamma)-D_1\right]
$$
where $D_1$ is the first coefficient in \eqref{interm4b}. Therefore, the flow on the center manifold is given by the following reduced system
\begin{equation}\label{flowgamma}
\left\{\begin{array}{ll}\dot{X}=\left[\frac{(m-1)\alpha(1-\gamma)}{\beta}-2\right]X^2+X^2O(|(X,H)|)\\
\dot{H}=\left[\sigma\gamma+k(1-\gamma)\right]X+\left[\sigma-k+\frac{k(1-\gamma)}{\gamma}\right]XH-DX^2+XO(|(X,H)|^2).\end{array}\right.
\end{equation}
with
$$
D=\frac{k(1-\gamma)\alpha[(N+\sigma-k)\beta+\alpha p(1-\gamma)]}{\beta^2}.
$$
We omit here the detailed calculations leading to all the previous expressions, since they can be done straightforwardly following the similar calculations in \cite[Lemma 2.4]{IS21a}. The main detail here is that all the terms are multiples of $X$. One can thus perform a change of the independent variable by letting
$$
d\theta=X d\eta,
$$
and the system \eqref{flowgamma} is topologically equivalent near the origin with the system obtained by dividing by $X$ its both equations. We then notice that $(X,H)=(0,0)$ remains a critical point and there are orbits connecting to it (with respect to the new variable $\theta$) if and only if the linear term in the equation for $\dot{H}$ vanishes. But this happens for
$$
\gamma=-\frac{1}{\alpha(p-1)}=\frac{\sigma(m-1)+2(p-1)}{(\sigma+2)(p-1)}<0,
$$
which is not in our range of interest. Thus, no critical point $P^{\gamma}$ with $\gamma>0$ contains orbits entering or going out of it on the center manifold. Finally, the orbits on the one-dimensional stable manifold enter $P^{\gamma}$ tangent to the eigenvector $e_2=(0,-1,k)$ corresponding to the only non-zero eigenvalue $-\beta/\alpha$ and it is easy to see that they are contained in the invariant plane $\{X=0\}$.
\end{proof}

\noindent \textbf{Changes for $p=1$.} With the above, we have completed the analysis of the finite critical points for $p>1$. In the case $p=1$, the proof of Lemma \ref{lem.Pgamma} becomes much easier as $k=0$, but there is a change with respect to the critical point $P_1$, which now expands into a critical line
$$
P_1^{\gamma}=\left(0,-\frac{\beta}{\alpha},\gamma\right), \qquad \gamma>0.
$$
\begin{lemma}[Analysis of the points $P_1^{\gamma}$ for $p=1$]\label{lem.P1p1}
For any $\gamma>0$, the critical point $P_1^{\gamma}$ has a one-dimensional stable manifold, a one-dimensional unstable manifold and a one-dimensional center manifold. The orbits entering $P_1^{\gamma}$ on the stable manifold contain profiles with interface behaving as in \eqref{beh.P1}.
\end{lemma}
We omit here the proof, as it is given in \cite[Lemma 2.2]{IS19}, where the interface point of the profile entering $P_1^{\gamma}$ is related to $\gamma$ by $\xi_0=(\alpha\gamma)^{1/\sigma}$.

\section{Local analysis at infinity in dimension $N\geq2$}\label{sec.infty}

This section is devoted to the local analysis of the critical points of the system \eqref{PSsyst1} at infinity, completing thus the classification of the critical points and, consequently, of the possible behaviors for the profiles $f(\xi)$ solutions to Eq. \eqref{ODE}. We follow the theory in \cite[Section 3.10]{Pe} by passing to the Poincar\'e hypersphere through the new variables $(\overline{X},\overline{Y},\overline{Z},W)$ defined as
$$
X=\frac{\overline{X}}{W}, \ Y=\frac{\overline{Y}}{W}, \ Z=\frac{\overline{Z}}{W}
$$
and we find from standard theory \cite[Theorem 4, Section 3.10]{Pe} that the critical points at space infinity lie on the equator of the hypersphere, thus at points
$(\overline{X},\overline{Y},\overline{Z},0)$ where $\overline{X}^2+\overline{Y}^2+\overline{Z}^2=1$ and the following system is satisfied:
\begin{equation}\label{Poincare}
\left\{\begin{array}{ll}\overline{X}(\overline{X}\overline{Z}+(N-2)\overline{X}\overline{Y}+m\overline{Y}^2)=0,\\
\overline{X}\overline{Z}[(\sigma+2)\overline{X}+(p-m)\overline{Y}]=0,\\
\overline{Z}[p\overline{Y}^2+(\sigma+N)\overline{X}\overline{Y}+\overline{X}\overline{Z}]=0,\end{array}\right.
\end{equation}
Taking into account that we are considering only points with coordinates $\overline{X}\geq0$ and $\overline{Z}\geq0$ and that we are working in dimension $N\geq2$, we find the following critical points on the Poincar\'e hypersphere:
\begin{equation*}
\begin{split}
&Q_1=(1,0,0,0), \ \ Q_{2,3}=(0,\pm1,0,0), \ \ Q_4=(0,0,1,0), \\
&Q_5=\left(\frac{m}{\sqrt{(N-2)^2+m^2}},-\frac{N-2}{\sqrt{(N-2)^2+m^2}},0,0\right).
\end{split}
\end{equation*}
We stress here that in dimension $N=1$ things are a bit more complicated, as in some range of exponents a new critical point will appear, but this is considered in Section \ref{sec.dim1} at the end of the paper. We give below the detailed analysis of the critical points listed above.

\medskip

\noindent \textbf{Analysis near $Q_1$ and $Q_5$ in dimension $N\geq3$}. In order to analyze the flow of the system \eqref{PSsyst1} locally near these points, we translate them into the finite part of a phase space associated to a new system, following the theory given in \cite[Theorem 5(a), Section 3.10]{Pe}. In our case, this system becomes
\begin{equation}\label{systinf1}
\left\{\begin{array}{ll}\dot{y}=-(N-2)y-z+w-my^2-\frac{\beta}{\alpha}yw,\\
\dot{z}=(\sigma+2)z+(p-m)yz,\\
\dot{w}=2w-(m-1)yw,\end{array}\right.
\end{equation}
where
\begin{equation}\label{interm6}
y=\frac{Y}{X}, \qquad z=\frac{Z}{X}, \qquad w=\frac{1}{X},
\end{equation}
and the minus sign has been chosen in the general framework of \cite[Theorem 5, Section 3.10]{Pe} according to the direction of the flow. Indeed, since in the equation for $\dot{X}$ in the original system \eqref{PSsyst1} we readily get $\dot{X}<0$ in a neighborhood of $Q_1$, as $|X/Y|\to+\infty$ near this point, we have to choose the minus sign in the general theory, leading to \eqref{systinf1}. We thus notice that $Q_1$ is mapped into the critical point $(y,z,w)=(0,0,0)$ and $Q_5$ into the critical point $(y,z,w)=(-(N-2)/m,0,0)$ of the system \eqref{systinf1}.
\begin{lemma}[Local analysis near $Q_1$ for $N\geq3$]\label{lem.Q1}
Let $N\geq3$. The system \eqref{systinf1} in a neihghborhood of $Q_1$ has a two-dimensional unstable manifold and a one-dimensional stable manifold. The orbits going out on the unstable manifold contain profiles with the local behavior \eqref{beh.Q1} as $\xi\to0$.
\end{lemma}
\begin{proof}
The linerization of the system \eqref{systinf1} in a neighborhood of $(y,z,w)=(0,0,0)$ has the matrix
$$
M(Q_1)=\left(
         \begin{array}{ccc}
           -(N-2) & -1 & 1 \\
           0 & \sigma+2 & 0 \\
           0 & 0 & 2 \\
         \end{array}
       \right),
$$
with two positive eigenvalues and one negative eigenvalue. The stable manifold is contained in the plane $\{w=0\}$, corresponding to $X=\infty$. The orbits going out on the unstable manifold satisfy $dz/dw\sim(\sigma+2)z/2w$, which after integration reads $z\sim Cw^{(\sigma+2)/2}$. This means that $Z\sim KX^{-\sigma/2}$, $K>0$, which in terms of profiles gives
$$
\xi^{\sigma}f(\xi)^{p-1}\sim K\xi^{\sigma}f(\xi)^{m-1},
$$
hence $f(\xi)\sim K$ for some $K>0$, and this asymptotic behavior is taken as $\xi\to0$ since $X\to\infty$ at $Q_1$. Moreover, since $z/w\sim Cw^{\sigma/2}\to+\infty$ in a neighborhood of $Q_1$, we can in a first approximation step neglect the $w$ term in the first equation in \eqref{systinf1} and find the following asymptotic approximation
$$
\frac{dy}{dz}\sim-\frac{N-2}{\sigma+2}\frac{y}{z}-\frac{1}{\sigma+2},
$$
which by integration leads to
$$
y\sim-\frac{z}{N+\sigma}+Kz^{-(N-2)/(\sigma+2)}.
$$
Since we want to pass through the point $(y,z)=(0,0)$, we have to take $K=0$ in the above estimate for $N\geq3$, obtaining $y\sim-z/(N+\sigma)$, whence $Y/Z\sim-1/(N+\sigma)$. Translating this in terms of profiles via \eqref{PSchange} and integrating the resulting differential equation, we are led to the local behavior given by \eqref{beh.Q1} as $\xi\to0$.
\end{proof}

\noindent \textbf{Remark.} The same local estimate for the behavior in terms of the phase space variables $Y/Z\sim-1/(N+\sigma)$ gives the following local behavior as $\xi\to0$ for the derivative of a profile $f(\xi)$ on the orbits going out of $Q_1$:
\begin{equation}\label{beh.deriv}
(f^m)'(\xi)\sim-\frac{1}{N+\sigma}\left[D-\frac{m-p}{m(N+\sigma)(\sigma+2)}\xi^{\sigma+2}\right]^{p/(m-p)}\xi^{\sigma+1}.
\end{equation}
This local estimate will be very important in the proof of the uniqueness part of Theorem \ref{th.1}.

\begin{lemma}[Local analysis near $Q_5$ for $N\geq3$]\label{lem.Q5}
Let $N\geq3$. The critical point $Q_5$ is an unstable node. The orbits going out of it into the finite part of the phase space contain profiles with a vertical asymptote at $\xi=0$ and the local behavior
\begin{equation}\label{beh.Q5}
f(\xi)\sim C\xi^{-(N-2)/m}, \qquad {\rm as} \ \xi\to0, \qquad C>0 \ {\rm free \ constant}.
\end{equation}
\end{lemma}
\begin{proof}
The linearization of the system \eqref{systinf1} in a neighborhood of $Q_5$ has the matrix
$$
M(Q_5)=\left(
         \begin{array}{ccc}
           N-2 & -1 & 1+\frac{(N-2)\beta}{m\alpha} \\
           0 & \sigma+2+\frac{(m-p)(N-2)}{m} & 0 \\
           0 & 0 & 2+\frac{(m-1)(N-2)}{m} \\
         \end{array}
       \right)
$$
with three positive eigenvalues for $N\geq3$. The local behavior is given by the fact that $y\to-(N-2)/m$, which in terms of profiles translates into
$$
\frac{f'(\xi)}{f(\xi)}\sim-\frac{N-2}{m\xi},
$$
which gives the local behavior \eqref{beh.Q5}. In a neighborhood of $Q_5$ we also have that $Z/X\to0$, that is,
\begin{equation}\label{interm5}
\xi^{\sigma+2}f(\xi)^{p-m}\sim\xi^{[m(N+\sigma)-p(N-2)]/m}\to0,
\end{equation}
thus the local behavior \eqref{beh.Q5} is taken in the limit as $\xi\to0$, since the exponent in \eqref{interm5} is positive.
\end{proof}

\medskip

\noindent \textbf{Analysis near $Q_1$ in dimension $N=2$.} We observe that for $N=2$, the points $Q_1$ and $Q_5$ coincide, and we will keep the label $Q_1$ for this mixed point.
\begin{lemma}[Local analysis near $Q_1$ for $N=2$]\label{lem.Q15N2}
Let $N=2$. Then the critical point $Q_1=Q_5$ is a saddle-node in the sense of the theory in \cite[Section 3.4]{GH}. There exists a two-dimensional unstable manifold on which the orbits contain good profiles with local behavior given by \eqref{beh.Q1} as $\xi\to0$. All the rest of the orbits going out of $Q_1$ contain profiles with a vertical asymptote at $\xi=0$ given by
\begin{equation}\label{beh.Q12}
f(\xi)\sim K\left(-\ln\,\xi\right)^{1/m}, \qquad {\rm as} \ \xi\to0, \ K>0 \ {\rm free \ constant}.
\end{equation}
\end{lemma}
\begin{proof}
The first linear term in the equation for $\dot{y}$ in the system \eqref{systinf1} vanishes if $N=2$. It readily follows that at $N=2$ the critical point $Q_1=Q_5$ is a saddle-node with a matrix having eigenvalues $\lambda_1=0$, $\lambda_2=\sigma+2$ and $\lambda_3=2$. We are mainly interested in the two-dimensional unstable manifold tangent to the vector space spanned by the eigenvectors $e_2=(-1,\sigma+2,0)$ and $e_3=(1,0,2)$ corresponding to the eigenvalues $\lambda_2$ and $\lambda_3$, which contains orbits whose analysis leads to the local behavior \eqref{beh.Q1} similarly as in Lemma \ref{lem.Q1}.

All the other orbits, according to the theory in \cite[Section 3.4]{GH}, go out tangent to the direction of the eigenvector corresponding to the eigenvalue $\lambda=0$ which is $e_1=(1,0,0)$. Moreover, we have $|y/w|\to\infty$, $|y/z|\to\infty$ on these orbits in a neighborhood of $Q_1$. The first equation in \eqref{systinf1} gives then locally that $\dot{y}\sim-my^2<0$, which proves that the orbits go out into the region $\{y<0\}$, which coincides with $\{Y<0\}$ in a neighborhood of $Q_1$. It remains to establish that the profiles contained in these orbits have a local behavior given by \eqref{beh.Q12}. This is done by inspection of Eq. \eqref{ODE}, proving that in a neighborhood of the point $Q_1$ along these orbits, the last three terms in Eq. \eqref{ODE} are negligible with respect to the terms involving $(f^m)''$ and $(f^m)'$ and thus the local behavior \eqref{beh.Q12} will be obtained by equating (in a first approximation)
$$
(f^m)''(\xi)+\frac{1}{\xi}(f^m)'(\xi)\sim 0.
$$
We do not give here the details of this part, as it is identical to the proof of \cite[Lemma 3.5]{IMS21} (see also \cite[Lemma 3.4]{IS21b}).
\end{proof}
In fact, if we let $N$ to be a parameter of the system \eqref{systinf1}, we are dealing with a \emph{transcritical bifurcation} of this system at $N=2$, in the sense of \cite{S73} (see also \cite[Section 3.4]{GH}).

\medskip

\noindent \textbf{Analysis near $Q_2$ and $Q_3$.} We translate these points into the finite part of a phase space associated to a new system, following the theory given in \cite[Theorem 5(b), Section 3.10]{Pe}. In our case, this system becomes
\begin{equation}\label{systinf2}
\left\{\begin{array}{ll}\pm\dot{x}=-mx-(N-2)x^2-\frac{\beta}{\alpha}xw+x^2w-x^2z,\\
\pm\dot{z}=-pz-\frac{\beta}{\alpha}zw-(N+\sigma)xz-xz^2+xzw,\\
\pm\dot{w}=-w-\frac{\beta}{\alpha}w^2+xw^2-Nxw-xzw,\end{array}\right.
\end{equation}
where the signs have to be chosen according to the direction of the flow. It is quite obvious that in a neighborhood of $Q_2$ one has to choose the minus sign in \eqref{systinf2}, while in a neighborhood of $Q_3$ one has to choose the plus sign, since $\dot{Y}$ is negative near both $Q_2$ and $Q_3$ but the direction of the flow is reversed. We thus identify $Q_2$ with the origin of \eqref{systinf2} when taken the minus sign and $Q_3$ with the origin of \eqref{systinf2} when taken the plus sign. It readily follows that $Q_2$ is an unstable node, while $Q_3$ is a stable node. Moreover, we have
\begin{lemma}[Local analysis near $Q_2$ and $Q_3$]\label{lem.Q23}
The orbits going out of $Q_2$ to the finite part of the phase space contain profiles $f(\xi)$ which change sign at some $\xi_0\in(0,\infty)$ in the sense that $f(\xi_0)=0$, $(f^m)'(\xi_0)>0$. The orbits entering the point $Q_3$ from the finite part of the phase space contain profiles $f(\xi)$ which change sign at some $\xi_0\in(0,\infty)$ in the sense that $f(\xi_0)=0$, $(f^m)'(\xi_0)<0$.
\end{lemma}
We omit here the proof of this Lemma, since it follows the proof of the similar Lemma in previous works such as for example \cite[Lemma 2.6]{IS21a}.

\medskip

\noindent \textbf{No profiles in orbits connecting to $Q_4$.} With respect to the remaining critical point $Q_4$, the local analysis following \cite[Theorem 5(c),Section 3.10]{Pe} would be very hard, since the topologically equivalent system has a linearization with a matrix with all three eigenvalues equal to zero. However, we can go back to Eq. \eqref{ODE} and work in terms of profiles to prove the following result.
\begin{lemma}\label{lem.Q4}
There are no profiles $f(\xi)$ solutions to Eq. \eqref{ODE} contained in orbits entering or going out of $Q_4$.
\end{lemma}
\begin{proof}
Assume for contradiction that there are such profiles. An orbit connecting to $Q_4$ satisfies the conditions $Z\to\infty$, $X/Z\to0$, $Y/Z\to0$, which translated in terms of profiles give
\begin{equation}\label{interm7}
\xi^{\sigma}f(\xi)^{p-1}\to\infty, \ \ \xi^{-(\sigma+2)}f(\xi)^{m-p}\to0, \ \ \xi^{-(1+\sigma)}f(\xi)^{m-p-1}f'(\xi)\to0,
\end{equation}
the limits in \eqref{interm7} being taken either as $\xi\to\infty$, or as $\xi\to\xi_0\in(0,\infty)$, or as $\xi\to0$. We rule out next all these three possibilities.

\medskip

\noindent \textbf{Step 1. No profiles with $\xi\to\infty$.} Assume for contradiction that the limits in \eqref{interm7} are taken as $\xi\to\infty$. Since $\xi^{\sigma}\to0$, it then follows that $f(\xi)^{p-1}\to\infty$, thus $f(\xi)\to\infty$. We prove first that there is $R>0$ large such that $f(\xi)$ is increasing for $\xi>R$. Indeed, if this is not the case, then there exists a sequence of local minima $\xi_{0,n}$ such that $\xi_{0,n}\to\infty$ as $n\to\infty$ and with $f'(\xi_{0,n})=0$, $f''(\xi_{0,n})\geq0$ for any $n\geq1$. We infer by evaluating Eq. \eqref{ODE} at $\xi=\xi_{0,n}$ that
$$
f(\xi_{0,n})\left[\xi_{0,n}^{\sigma}f(\xi_{0,n})^{p-1}-\alpha\right]\leq0,
$$
which is impossible since $f(\xi_{0,n})>0$ and the term in brackets tends to infinity according to the first limit in \eqref{interm7}. We thus get that $f(\xi)$ is increasing for $\xi>R$ with $R$ sufficiently large. We next deduce from Eq. \eqref{ODE} that
$$
(f^m)''(\xi)=-f(\xi)\left[\xi^{\sigma}f(\xi)^{p-1}-\alpha\right]-\beta\xi f'(\xi)-\frac{N-1}{\xi}(f^m)'(\xi)\to-\infty,
$$
as $\xi\to\infty$, which gives that for any $K>0$ sufficiently large, there exists some $\xi_0(K)>R$, depending on $K$, such that $(f^m)''(\xi)<-K$ for any $\xi>\xi_0(K)$. We fix such $K>0$ large, take $\xi_2>\xi_1>\xi_0(K)$ and find that
$$
(f^m)'(\xi_2)=(f^m)'(\xi_1)+\int_{\xi_1}^{\xi_2}(f^m)''(s)\,ds<(f^m)'(\xi_1)-K(\xi_2-\xi_1)<0
$$
provided $\xi_2$ is taken sufficiently large, and this is a contradiction with the fact that $f$ was increasing for $\xi>\xi_0(K)>R$.

\medskip

\noindent \textbf{Step 2. No profiles with $\xi\to\xi_0\in(0,\infty)$.} This is rather obvious from the first two limits in \eqref{interm7}, as if \eqref{interm7} holds true as $\xi\to\xi_0\in(0,\infty)$, then on the one hand $f(\xi)^{p-1}\to\infty$ and on the other hand $f(\xi)^{m-p}\to0$, which is contradictory since both $p-1$ and $m-p$ are positive powers.

\medskip

\noindent \textbf{Step 3. No profiles with $\xi\to0$.} Assume for contradiction that the limits in \eqref{interm7} are taken as $\xi\to0$. Since $m-p>0$, we readily obtain from the second limit in \eqref{interm7} that $f(\xi)\to0$ as $\xi\to0$, and a similar argument with local minima as in the first step of this proof shows that $f$ is increasing on a right-neighborhood $\xi\in(0,R)$ of the origin. We then let $\xi\in(0,R)$ and write Eq. \eqref{ODE} in the following form:
\begin{equation}\label{interm8}
(f^m)''(\xi)+\frac{1}{2}\xi^{\sigma}f(\xi)^p+\frac{N-1}{\xi}(f^m)'(\xi)+\beta\xi f'(\xi)+f(\xi)\left[\frac{1}{2}\xi^{\sigma}f(\xi)^{p-1}-\alpha\right]=0.
\end{equation}
Since all the last terms are positive for $\xi\in(0,R)$, we infer that
\begin{equation}\label{interm9}
(f^m)''(\xi)+\frac{1}{2}\xi^{\sigma}f(\xi)^p<0, \qquad {\rm for \ any} \ \xi\in(0,R).
\end{equation}
We can then write for $\xi\in(0,R)$
\begin{equation*}
\begin{split}
(f^m)''(\xi)+\frac{1}{2}\xi^{\sigma}f(\xi)^p&=\xi^{\sigma}f(\xi)^p\left[\frac{1}{2}+m\xi^{-\sigma}f^{m-p-1}(\xi)f''(\xi)\right]+m(m-1)f^{m-2}(\xi)(f'(\xi))^2\\
&\geq\xi^{\sigma}f(\xi)^p\left[\frac{1}{2}+m\xi^{-(1+\sigma)}f^{m-p-1}(\xi)\frac{\xi f''(\xi)}{f'(\xi)}\right]\\
&=\xi^{\sigma}f(\xi)^p\left[\frac{1}{2}+m\frac{Y}{Z}\frac{\xi f''(\xi)}{f'(\xi)}\right].
\end{split}
\end{equation*}
This together with \eqref{interm9} show that the term in brackets in the last line above has to be negative, and since $Y/Z\to0$, we further infer that
\begin{equation}\label{interm10}
\lim\limits_{\xi\to0}\frac{\xi f''(\xi)}{f'(\xi)}=-\infty.
\end{equation}
We deduce from \eqref{interm10} that, for any $K>0$, there exists $\xi(K)>0$ such that
$$
\frac{\xi f''(\xi)}{f'(\xi)}<-K, \qquad {\rm for \ any} \ \xi\in(0,\xi(K)),
$$
which leads to $(\xi^Kf'(\xi))'<0$, thus $\xi^Kf'(\xi)$ is decreasing on $(0,\xi(K))$. We obtain that $\liminf\limits_{\xi\to0}\xi^Kf'(\xi)>0$ for any $K>0$, which readily gives that in fact
$$
\lim\limits_{\xi\to0}\xi^Kf'(\xi)=+\infty, \qquad {\rm for \ any} \ K>0,
$$
and this is a contradiction to the integrability of $f'$ near $\xi=0$. 
\end{proof}
The local analysis is now completed and we can pass to the proof of Theorem \ref{th.1}.

\section{Proof of Theorem \ref{th.1}: existence for $N\geq2$}\label{sec.exist}

We perform a shooting technique on the two-dimensional manifold going out of $Q_1$ tangent to the plane spanned by the eigenvectors $e_2=(-1,N+\sigma,0)$ and $e_3=(1,0,N)$ corresponding to the eigenvalues $\lambda_2=\sigma+2$ and $\lambda_3=2$ of the matrix $M(Q_1)$ in Lemma \ref{lem.Q1} (with its respective changes for $N=2$ in Lemma \ref{lem.Q15N2}). We work with the system \eqref{systinf1} and notice that these eigenvectors lie on the invariant planes $\{w=0\}$, respectively $\{z=0\}$. Moreover, the trajectories of the two-dimensional manifold going out of $Q_1$ are tangent to the one-parameter family of curves given by
\begin{equation}\label{interm11}
z=Cw^{(\sigma+2)/2}, \qquad C\in(0,\infty).
\end{equation}
We explore the behavior of the limiting orbits of the manifold, contained in the invariant planes $\{w=0\}$, respectively $\{z=0\}$, in the following preparatory results.
\begin{proposition}\label{prop.lower}
There exists $C_{*}>0$ such that the orbits going out of $Q_1$ on the two-dimensional manifolds tangent to curves corresponding to $C\in(0,C_*)$ in \eqref{interm11} enter the critical point $P_0$.
\end{proposition}
\begin{proof}
We analyze the limit of the manifold contained in the invariant plane $\{z=0\}$, which corresponds to the limit $C=0$ in \eqref{interm11}. This limit is the unique trajectory going out of $Q_1$ inside the invariant plane $\{z=0\}$, tangent to the eigenvector $e_3$. The system \eqref{systinf1} reduces in the plane $\{z=0\}$ to the system
\begin{equation}\label{systz=0}
\left\{\begin{array}{ll}\dot{y}=-(N-2)y+w-my^2-\frac{\beta}{\alpha}yw,\\ \dot{w}=w(2-(m-1)y).\end{array}\right.
\end{equation}
In this reduced system, $Q_1$ becomes a saddle point and the unique orbit going out of $Q_1$ is tangent to the eigenvector $e_3=(1,N)$, thus it enters the region $\{y>0\}$. Since the direction of the flow of the system \eqref{systz=0} on the line $\{y=0\}$ is given by the sign of $w>0$, it follows that this orbit will stay forever in the region $\{y>0\}$ of the plane, which also corresponds to $\{Y>0\}$ in the variables of \eqref{PSchange}. Consider now the isocline
\begin{equation}\label{isoz=0}
-(N-2)y+w-my^2-\frac{\beta}{\alpha}yw=0, \qquad {\rm or \ equivalently} \ w=\frac{my^2+(N-2)y}{1-(\beta/\alpha)y}.
\end{equation}
Since the orbit from $Q_1$ goes out with both $y$ and $w$ increasing, it is easy to see that it enters near $Q_1$ the region limited by the line $\{y=0\}$ and the isocline \eqref{isoz=0}. This region is plotted, together with the isoclines, in Figure \ref{fig2}, where the plus or minus signs indicate the monotonicity of the dependence $w=w(y)$ along the trajectories.

\begin{figure}[ht!]
  \begin{center}
  \includegraphics[width=11cm,height=8cm]{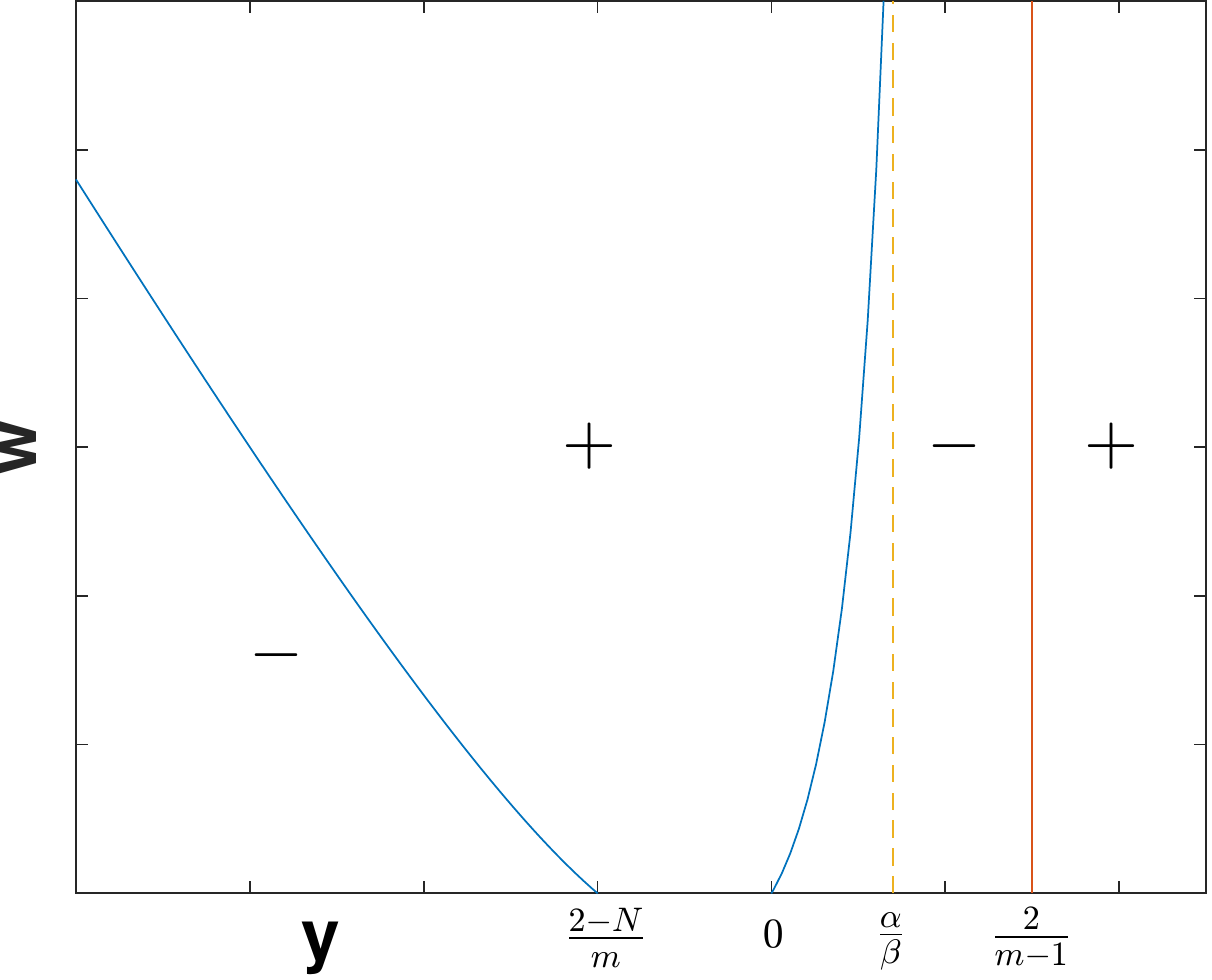}
  \end{center}
  \caption{The isoclines and monotonicity regions of the phase plane associated to the system \eqref{systz=0}}\label{fig2}
\end{figure}

Moreover, in this region we always have $w>0$, which leads to $y<\alpha/\beta<2/(m-1)$ and we infer that the orbit cannot cross the isocline \eqref{isoz=0} later and will thus stay forever inside this region. We further observe that
$$
2-(m-1)y>0, \qquad w>\frac{my^2+(N-2)y}{1-(\beta/\alpha)y},
$$
hence both $w$ and $y$ are increasing along the trajectory. We moreover notice that the estimate $y<\alpha/\beta$ implies, with respect to the initial variables given by \eqref{PSchange}, that
\begin{equation}\label{barrier}
0<Y<\frac{\alpha}{\beta}X. 
\end{equation}
We then get that the orbit cannot enter a limit cycle and has to end up in a critical point (with $w\to\infty$ in this case, which in initial variables means $X=1/w=0$) of the reduced system \eqref{systz=0}, which is also a critical point of the system \eqref{PSsyst1} with $X=0$. The estimate \eqref{barrier} proves that this critical point is the attractor $P_0$. Since $P_0$ is an attractor according to Lemma \ref{lem.P0}, by a standard continuity argument it follows that the orbits going out of $Q_1$ in the manifold tangent to the curves \eqref{interm11} also enter $P_0$ for sufficiently small parameters $C>0$, as claimed.
\end{proof}
\begin{proposition}\label{prop.higher}
There exists $C^{*}>0$ such that the orbits going out of $Q_1$ on the two-dimensional manifold tangent to curves corresponding to $C\in(C^*,+\infty)$ in \eqref{interm11} enter the critical point $Q_3$.
\end{proposition}
\begin{proof}
We now analyze the limit of the manifold from $Q_1$ contained in the invariant plane $\{w=0\}$, which corresponds to the limit $C\to\infty$ in \eqref{interm11}. The system \eqref{systinf1} reduces in the plane $\{w=0\}$ to the system
\begin{equation}\label{systw=0}
\left\{\begin{array}{ll}\dot{y}=-(N-2)y-my^2-z,\\ \dot{z}=z(\sigma+2-(m-p)y).\end{array}\right.
\end{equation}
In this reduced system, $Q_1$ becomes a saddle point and the unique orbit going out of $Q_1$ is tangent to the eigenvector $e_2=(-1,N+\sigma)$, thus it enters the region $\{y<0\}$ and it cannot cross the line $\{y=0\}$ later on, as the direction of the flow on this line is negative. Consider now the isocline
\begin{equation}\label{isow=0}
-(N-2)y-my^2-z=0,
\end{equation}
whose slope near $(y,z)=(0,0)$ is $z/y\sim-(N-2)$. Since the eigenvector tangent to the trajectory going out of $Q_1$ has the slope $z/y=-(N+\sigma)<-(N-2)$, we deduce that the orbit going out of $Q_1$ enters the region limited by the line $\{y=0\}$ and by the isocline \eqref{isow=0}. Moreover, the direction of the flow on the isocline \eqref{isow=0} is negative and the normal vector to it is $\overline{n}=(-(N-2)-2my,-1)$, hence the isocline cannot be crossed from the region $\{z>-(N-2)y-my^2\}$ towards the region interior to it. We thus find that the orbit going out of $Q_1$ will remain forever in the region $\{y<0,z>-(N-2)y-my^2\}$ of the plane associated to the system \eqref{systw=0}. For the reader's convention, we plot the above mentioned regions in Figure \ref{fig3}, where the plus and minus signs indicate the monotonicity regions for the dependence $z=z(y)$ in the phase plane associated to the system \eqref{systw=0}.

\begin{figure}[ht!]
  \begin{center}
  \includegraphics[width=11cm,height=8cm]{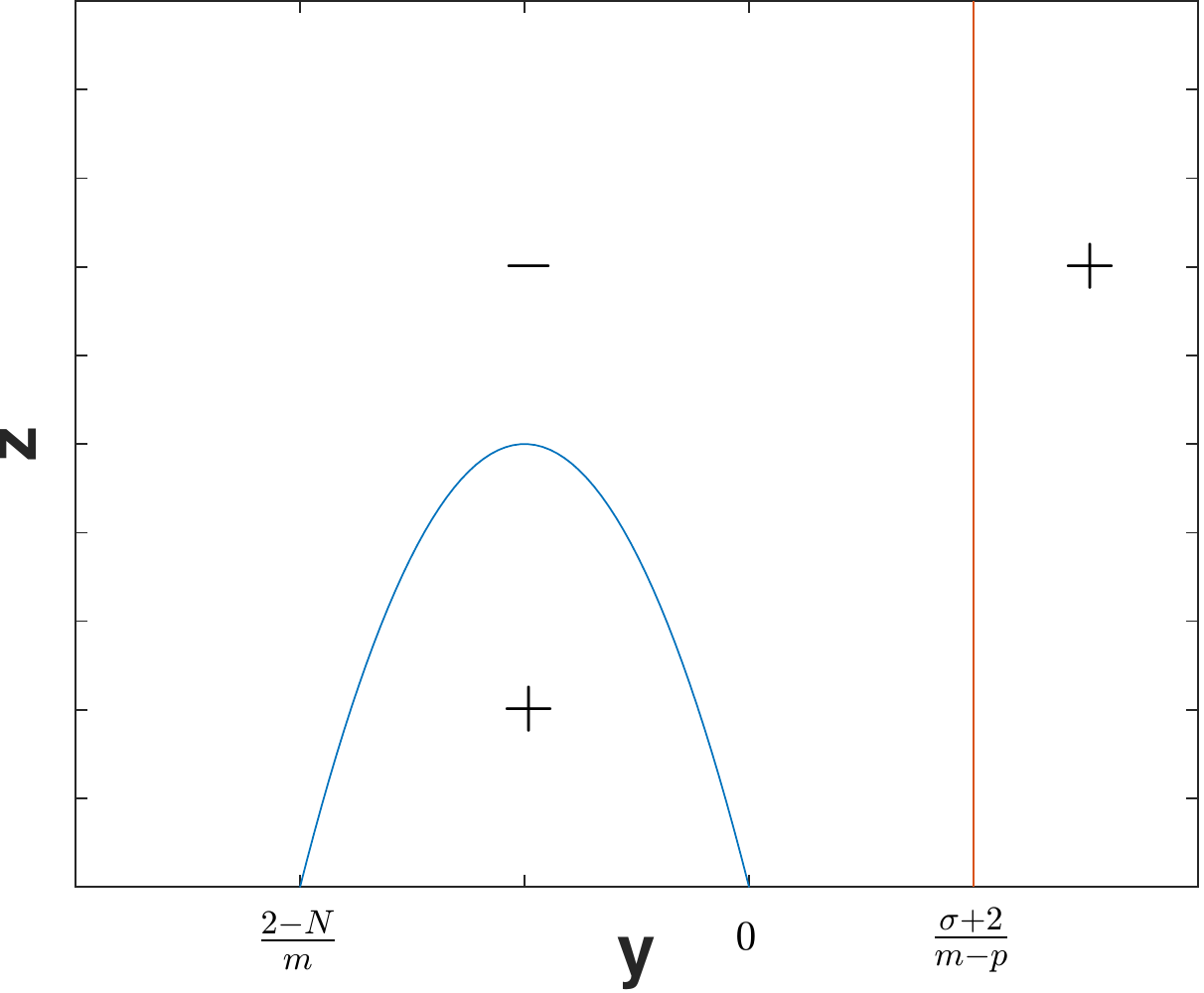}
  \end{center}
  \caption{The isoclines and monotonicity regions of the phase plane associated to the system \eqref{systw=0}}\label{fig3}
\end{figure}

We further observe that in this region, variable $y$ is decreasing and variable $z$ is increasing along the trajectories, thus they have a limit. We show below that $y\to-\infty$ and $z\to\infty$ along this trajectory. Indeed, if this is not the case, we would be in one of the following three situations:

$\bullet$ finite limits $y\to y_0\in(-\infty,0)$, $z\to z_0\in(0,\infty)$. Then $(y_0,z_0)$ must be a finite critical point of the system \eqref{systw=0} and there is no such point.

$\bullet$ if $y\to y_0\in(-\infty,0)$ and $z\to\infty$, we get 
$$
\frac{dy}{dz}=-\frac{(N-2)y+my^2+z}{z(\sigma+2-(m-p)y)}\to\frac{1}{\sigma+2-(m-p)y_0}>0
$$
and this is an obvious contradiction with the existence of such a vertical asymptote of the trajectory.

$\bullet$ if $y\to-\infty$ and $z\to z_0\in(0,\infty)$, we get for $y$ and $z$ very large in absolute value that 
$$
\frac{dy}{dz}\sim\frac{my}{(m-p)y_0}, \qquad {\rm hence} \qquad y\sim Ke^{mz/(m-p)z_0}
$$
and this is an obvious contradiction with the existence of such a horizontal asymptote of the trajectory.

We thus conclude that $y\to-\infty$ and $z\to\infty$ along this trajectory. We then get that for $y$ and $z$ very large in absolute value 
$$
\frac{dy}{dz}\sim\frac{m}{m-p}\frac{y}{z}+\frac{1}{(m-p)y},
$$
and we infer by integration that 
\begin{equation}\label{interm15}
y^2\sim Kz^{2m/(m-p)}-\frac{2z}{m+p}, \qquad K>0.
\end{equation}
Since $2m/(m-p)>1$, it follows that for $z$ large, the first term in \eqref{interm15} is dominating and we are left with 
$$
y\sim-Kz^{m/(m-p)}, \qquad K>0.
$$
Since $m/(m-p)>1$, we notice that $y/z\to-\infty$. Going back to the initial variables in \eqref{PSchange}, we deduce that 
$$
\frac{Y}{Z}\to-\infty, \qquad \frac{Y}{X}\to-\infty
$$
along this trajectory. We conclude that the orbit enters the stable node $Q_3$. Since $+\infty$ is not a number in order to use a continuity argument near it for the conclusion, in order to be rigorous we can express the curves in \eqref{interm11} also in the form
$$
w=\left(\frac{1}{C}z\right)^{2/(\sigma+2)}=C_1z^{2/(\sigma+2)}
$$
and take a neighborhood of $C_1=0$, allowing thus to use the same argument as at the end of the proof of Proposition \ref{prop.lower} here.
\end{proof}
Before proving the existence part in Theorem \ref{th.1}, we need one more preparatory lemma which will be very important both in the classification of the profiles and in the proof of the uniqueness.
\begin{lemma}\label{lem.cross}
The components $X$ and $Z$ are decreasing along the orbits going out of $Q_1$.
\end{lemma}
\begin{proof}
On the one hand, the direction of the flow of the system \eqref{PSsyst1} on the plane $(m-1)Y-2X=0$ is given by the sign of the expression
$$
F(X,Z)=-\frac{2(mN-N+2)}{m-1}X^2-(m-1)XZ+\frac{\sigma(m-1)+2(p-1)}{\sigma+2}X<0,
$$
since $\sigma(m-1)+2(p-1)<0$. On the other hand, the direction of the flow of the system \eqref{PSsyst1} on the plane $(p-1)Y+\sigma X$ is given by the sign of the expression
$$
G(X,Z)=\frac{\sigma[(N-2)(p-1)-m\sigma]}{p-1}X^2-(p-1)XZ+\frac{\sigma(m-1)+2(p-1)}{\sigma+2}X<0,
$$
since $\sigma<0$. It thus follows that none of these two planes can be crossed in the positive direction with respect to the variable $Y$. Since the orbits going out of $Q_1$ on the manifold tangent to the curves in \eqref{interm11} enter the half-space $\{Y<0\}$, they will never cross these two planes and thus will remain forever in the region where both $(m-1)Y-2X<0$ and $(p-1)Y+\sigma X<0$, which means $\dot{X}<0$, $\dot{Z}<0$ along the trajectories, completing the proof.
\end{proof}
\begin{proof}[Proof of Theorem \ref{th.1}, existence part]
We apply the "three-sets method". Making an abuse of language, let us identify the manifold going out of $Q_1$ with the one-parameter family of curves given by \eqref{interm11}. We then consider the following three sets
\begin{equation*}
\begin{split}
&\mathcal{A}=\{C\in(0,\infty): {\rm the \ orbit \ with \ parameter} \ C \ {\rm in \ \eqref{interm11} \ enters} \ P_0\},\\
&\mathcal{C}=\{C\in(0,\infty): {\rm the \ orbit \ with \ parameter} \ C \ {\rm in \ \eqref{interm11} \ enters} \ Q_3\},\\
&\mathcal{B}=\{C\in(0,\infty): {\rm the \ orbit \ with \ parameter} \ C \ {\rm in \ \eqref{interm11} \ does \ neither \ enter} \ P_0 \ {\rm nor} \ Q_3\}.
\end{split}
\end{equation*}
Propositions \ref{prop.lower} and \ref{prop.higher} show that the sets $\mathcal{A}$ and $\mathcal{C}$ are non-empty, while standard continuity arguments give that both $\mathcal{A}$ and $\mathcal{C}$ are open, since $P_0$ and $Q_3$ are attractors. It thus follows that $\mathcal{B}$ is non-empty. Take now $C\in\mathcal{B}$. The orbit tangent to the curve in \eqref{interm11} with parameter $C$ does not enter either $P_0$ or $Q_3$, and Lemma \ref{lem.cross} ensures that it cannot also enter a limit cycle since components $X$ and $Z$ are decreasing along it. It thus has to enter the \emph{only} critical point where this is allowed, that is $P_1$. We thus find at least one orbit connecting $Q_1$ to $P_1$, and Lemmas \ref{lem.Q1} and \ref{lem.P1} prove that the profiles contained in such orbits satisfy the statement of Theorem \ref{th.1}.
\end{proof}
Before passing to the uniqueness part, let us translate the previous results in terms of profiles.
\begin{corollary}\label{cor.shoot}
The profiles starting with the behavior \eqref{beh.Q1} for some $D>0$ have only two possibilities with respect to their monotonicity:

$\bullet$ they are either decreasing on their support $\{\xi>0: f(\xi)>0\}$

$\bullet$ or they decrease on some interval $\xi\in(0,\xi_0)$ with an absolute minimum $f(\xi_0)>0$ and then increase for $\xi\in(\xi_0,\infty)$.

Moreover, there exist $D_*>0$ and $D^*>0$ such that

$\bullet$ the profiles as in \eqref{beh.Q1} with $D\in(0,D_*)$ have compact support and a change of sign at a point $\xi_0\in(0,\infty)$, in the sense of Lemma \ref{lem.Q23}.

$\bullet$ the profiles as in \eqref{beh.Q1} with $D\in(D^*,\infty)$ have a unique positive minimum and afterwards they increase to infinity.
\end{corollary}
\begin{proof}
The flow on the plane $\{Y=0\}$ is given by the sign of the expression $X(1-Z)$. Profiles as in \eqref{beh.Q1} are contained in orbits going out of $Q_1$ and entering the half-space $\{Y<0\}$ of the phase space associated to the system \eqref{PSsyst1}. Such orbits may either remain forever in the half-space $\{Y<0\}$ (thus, profiles are decreasing and going to either $Q_3$ or $P_1$, hence compactly supported) or cross the plane $\{Y=0\}$ from left to right, which has to be done with $Z<1$. The crossing point corresponds to a point of minimum $\xi_0>0$ for the profile. Since component $Z$ is always decreasing along the orbit by Lemma \ref{lem.cross}, we readily infer that such an orbit can cross $\{Y=0\}$ only once, thus remaining later in the half-space $\{Y>0\}$ until entering $P_0$, and this gives the second alternative. For the second statement, we remark that \eqref{interm11} is equivalent to
$$
f^{m-p}(\xi)\sim\frac{1}{C}, \qquad {\rm as} \ \xi\to0.
$$
We readily get the conclusion from the intervals in Propositions \ref{prop.lower} and \ref{prop.higher} and the local behaviors at $P_0$ and $Q_3$ in Lemmas \ref{lem.P0} and \ref{lem.Q23}.
\end{proof}

\section{Proof of Theorem \ref{th.1}: uniqueness for $N\geq2$}\label{sec.uniq}

We are now able to prove the uniqueness part of Theorem \ref{th.1}. This will be done working directly with Eq. \eqref{ODE} and using a \emph{sliding technique}, which leads to the construction of optimal barriers. The technical core of the argument is the following monotonicity result with respect to the value $f(0)$, whose proof is partly inspired by the one in \cite{YeYin}. We recall first that there are no profiles with $f(\xi)\to0$ as $\xi\to\infty$, according to the classification performed in Sections \ref{sec.local} and \ref{sec.infty}.
\begin{lemma}[Monotonicity Lemma]\label{lem.monot}
Let $f_1$, $f_2$ be two profiles as in \eqref{beh.Q1} and such that $f_1(0)=D_1<D_2=f_2(0)$. Let $\xi_{\rm min}\in(0,\infty)$ be such that either $f_1(\xi_{\rm min})=0$ or such that $f_1(\xi_{\rm min})>0$ is the positive minimum of $f_1$, according to the alternative in Corollary \ref{cor.shoot}. Then we have $f_1(\xi)<f_2(\xi)$ for any $\xi\in(0,\xi_{\rm min})$.
\end{lemma}
\begin{proof}
We work with the differential equation satisfied by $g=f^m$, that is
\begin{equation}\label{ODE2}
g''(\xi)+\frac{N-1}{\xi}g'(\xi)-\alpha g(\xi)^{1/m}+\beta\xi(g^{1/m})'(\xi)+\xi^{\sigma}g(\xi)^{p/m}=0.
\end{equation}
and set $g_1=f_1^m$, $g_2=f_2^m$. Assume for contradiction that there exists a first point of intersection between $g_1$ and $g_2$, namely $\xi_0\in(0,\xi_{\rm min})$, such that
$g_1(\xi)<g_2(\xi)$ for any $\xi\in[0,\xi_0)$ and $g_1(\xi_0)=g_2(\xi_0)$. Let us also simplify the notation by setting
$$
L=\sigma(m-1)+2(p-1)<0 \qquad {\rm for} \ 1<p<p_c(\sigma).
$$
We divide the rest of the proof into several steps.

\medskip

\noindent \textbf{Step 1. Rescaling and optimal barrier.} Introduce for any $k>0$ the following rescaling
\begin{equation}\label{resc}
f_k(\xi)=k^{-2/(m-1)}f_1(k\xi), \qquad g_k(\xi)=k^{-2m/(m-1)}g_1(k\xi).
\end{equation}
On the one hand, we find by direct calculation that the rescaled functions $g_k$ solve the differential equation
\begin{equation}\label{ODEresc}
g_k''(\xi)+\frac{N-1}{\xi}g_k'(\xi)-\alpha g_k(\xi)^{1/m}+\beta\xi(g_k^{1/m})'(\xi)+k^{L/(m-1)}\xi^{\sigma}g_k(\xi)^{p/m}=0.
\end{equation}
On the other hand, we notice that, if $0<k_1<k_2\leq1$ and $\xi\in[0,\xi_{\rm min}]$ we have $g_{k_1}(\xi)>g_{k_2}(\xi)$ and that
$$
\lim\limits_{k\to0}g_k(\xi)=+\infty, \qquad {\rm for \ any} \ \xi\in[0,\xi_{\rm min}).
$$
We thus infer that there exists $\overline{k}\in(0,1]$ such that $g_2(\xi)<g_{k}(\xi)$ for any $\xi\in[0,\xi_{\rm min}]$ and for any $k\in(0,\overline{k})$. This allows us to slide the rescaling parameter $k$ and consider the optimal one
\begin{equation}\label{opt.bar}
k_0=\sup\{\overline{k}>0: g_2(\xi)<g_k(\xi), \ {\rm for \ any} \ \xi\in[0,\xi_{\rm min}] \ {\rm and} \ k\in(0,\overline{k})\}.
\end{equation}
Since $g_1(\xi_0)=g_2(\xi_0)$ and $g_1(\xi)<g_2(\xi)$ for $\xi\in[0,\xi_0)$, it readily follows that $k_0<1$.

\medskip

\noindent \textbf{Step 2. No contact point at $\xi=\xi_0$.} We derive from the optimality of $k_0$ defined in \eqref{opt.bar} that the function $g_{k_0}$ should touch from above the function $g_2$, thus having a contact point $\xi_1\in[0,\xi_0]$. If $\xi_1=\xi_0$, this means $g_{k_0}(\xi_0)=g_2(\xi_0)=g_1(\xi_0)$, whence
$$
k_0^{-2m/(m-1)}g_{1}(k_0\xi_0)=g_1(\xi_0),
$$
which is a contradiction with the fact that $k_0<1$ and the decreasing monotonicity of $g_1$ on $[0,\xi_0]$. We thus get that $\xi_1<\xi_0$.

\medskip

\noindent \textbf{Step 3. No contact point in the interval $\xi\in(0,\xi_0)$.} Assume by contradiction that the touching point $\xi_1$ between the optimal barrier $g_{k_0}$ and $g_2$ belongs to the interval $(0,\xi_0)$. Thus, the two functions are tangent at $\xi=\xi_1$, which implies that $g_2(\xi_1)=g_{k_0}(\xi_1)$, $g_2'(\xi_1)=g_{k_0}'(\xi_1)$ and $g_{k_0}''(\xi_1)\geq g_2''(\xi_1)$. But we can also compare the two second derivatives by subtracting them from the differential equations \eqref{ODEresc} (solved by $g_{k_0}$) and \eqref{ODE2} (solved by $g_2$) taking into account the equality of $g_2$ and $g_{k_0}$ and also of their first derivatives at $\xi=\xi_1$. We thus find
$$
g_{k_0}''(\xi_1)-g_2''(\xi_1)=\left(1-k_0^{L/(m-1)}\right)\xi_1^{\sigma}g_2(\xi_1)^{p/m}<0,
$$
since $k_0<1$ and $L<0$. We thus reach a contradiction, proving that there is no such contact point $\xi_1\in(0,\xi_0)$.

\medskip

\noindent \textbf{Step 4. No contact point at $\xi=0$.} Assume for contradiction that the contact point $\xi_1=0$. It then follows that $g_{k_0}(0)=g_2(0)$ and $g_{k_0}(\xi)>g_2(\xi)$ for any $\xi\in(0,\xi_0)$ (at least), which also implies that there is a right-neighborhood $(0,\delta)$ of the origin such that $g_{k_0}'(\xi)>g_2'(\xi)$ for $\xi\in(0,\delta)$. We will show that this situation is not possible and thus reach a contradiction by using as a decisive argument the precise behavior \eqref{beh.Q1} of both $f_1$ and $f_2$ near $\xi=0$. Indeed, we deduce from \eqref{beh.Q1} and the definition of $g_1$, $g_2$ and $g_{k_0}$ (given in \eqref{resc}) that
\begin{equation}\label{interm12}
g_{k_0}(\xi)\sim k_0^{-2m/(m-1)}\left[D_1-Kk_0^{\sigma+2}\xi^{\sigma+2}\right]^{m/(m-p)}, \ \ g_2(\xi)\sim\left[D_2-K\xi^{\sigma+2}\right]^{m/(m-p)},
\end{equation}
where
$$
K=\frac{m-p}{m(N+\sigma)(\sigma+2)}>0.
$$
We infer from \eqref{interm12} that the equality $g_{k_0}(0)=g_2(0)$ gives
\begin{equation}\label{interm13}
D_2^{1/(m-p)}=k_0^{-2/(m-1)}D_1^{1/(m-p)}.
\end{equation}
We now analyze the local behavior of the derivatives of $g_{k_0}$ and $g_2$ as $\xi\to0$. We obtain from the general local estimate for derivatives \eqref{beh.deriv} that
$$
g_{k_0}'(\xi)\sim-\frac{1}{N+\sigma}k_0^{\sigma+2-2m/(m-1)}\left[D_1-Kk_0^{\sigma+2}\xi^{\sigma+2}\right]^{p/(m-p)}\xi^{\sigma+1},
$$
and
$$
g_2'(\xi)\sim-\frac{1}{N+\sigma}\left[D_2-K\xi^{\sigma+2}\right]^{p/(m-p)}\xi^{\sigma+1}
$$
as $\xi\to0$. After simplifying common factors, the condition $g_{k_0}'(\xi)>g_2'(\xi)$ for $\xi\in(0,\delta)$ leads by taking limit as $\xi\to0$ and recalling \eqref{interm13} to
$$
-k_0^{\sigma+2-2m/(m-1)}D_1^{p/(m-p)}>-D_2^{p/(m-p)}=-k_0^{2p/(m-1)}D_1^{p/(m-p)},
$$
or equivalently
$$
k_0^{-L/(m-1)}>1,
$$
which is a contradiction since $k_0<1$ and $-L/(m-1)>0$. Thus, there cannot be any contact point, which contradicts the existence of the intersection point $\xi_0$, proving the monotonicity.
\end{proof}
We deduce from Lemma \ref{lem.monot} that two profiles cannot intersect while both are decreasing. We are now almost ready to complete the proof of the uniqueness part, but we still need one more preparatory lemma. Let us take $f$ to be a profile contained in an orbit connecting the critical points $Q_1$ and $P_1$ in the phase space, thus with local behavior given by \eqref{beh.Q1} at $\xi=0$ and with an interface according to the local behavior \eqref{beh.P1} at some $\xi_0\in(0,\infty)$.
\begin{lemma}\label{lem.super}
If we define for some $k\in(0,1)$
$$
U_k(x,t)=t^{\alpha}f_k(|x|t^{-\beta}), \qquad f_k(\xi)=k^{-2/(m-1)}f(k\xi), \qquad f(\xi) \ {\rm solution \ to} \ \eqref{ODE},
$$
then $U_k$ is a supersolution to Eq. \eqref{eq1}.
\end{lemma}
\begin{proof}
By direct calculation we find that
\begin{equation*}
\begin{split}
U_{k,t}-\Delta U_k^m-|x|^{\sigma}U_k^p&=t^{\alpha-1}\left[\alpha f_k-\beta\xi f_k'-(f_k^m)''-\frac{N-1}{\xi}(f_k^m)'-\xi^{\sigma}f_k^p\right]\\
&=t^{\alpha-1}\xi^{\sigma}f_k^p\left(k^{L/(m-1)}-1\right)>0,
\end{split}
\end{equation*}
since $k<1$ and $L<0$. The conclusion follows from the fact that the local behavior at $\xi=0$ and the contact condition at the interface point remain valid with the rescaling.
\end{proof}
We are now in a position to complete the proof of Theorem \ref{th.1} in dimension $N\geq2$.
\begin{proof}[Proof of Theorem \ref{th.1}, uniqueness part for $N\geq2$]
Assume for contradiction that there exist two profiles $f_1$ and $f_2$ with behavior as in \eqref{beh.Q1} and with interfaces at points $\xi_1\in(0,\infty)$, respectively $\xi_2\in(0,\infty)$ such that $\xi_1<\xi_2$. Since Lemma \ref{lem.cross} implies that both profiles $f_1$ and $f_2$ are decreasing on all their support, we infer from Lemma \ref{lem.monot} that the two profiles are totally ordered over the positive part of the smallest one, that is, $f_1(\xi)<f_2(\xi)$ for any $\xi\in(0,\xi_1)$. Consider now the same rescaling introduced in \eqref{resc} and define $k_0$ to be the optimal parameter as in \eqref{opt.bar}, $k_0<1$. According to the proof of Lemma \ref{lem.monot}, we already know that $g_{k_0}$ and $g_2$, thus also $f_{k_0}$ and $f_2$, cannot have a contact point at some $\xi\in[0,\xi_2)$. We get that the unique remaining possibility is that the contact point between $f_{k_0}$ and $f_2$ has to lie at their common edge of the support $\xi=\xi_2$. We thus have
$$
0=f_2(\xi_2)=f_{k_0}(\xi_2), \qquad 0<f_2(\xi)<f_{k_0}(\xi) \ {\rm for \ any} \ \xi\in[0,\xi_2).
$$
We go back to the time variable and construct the self-similar functions
$$
U_2(x,t)=t^{\alpha}f_2(|x|t^{-\beta}), \qquad U_{k_0}=t^{\alpha}f_{k_0}(|x|t^{-\beta}),
$$
thus $U_2$ is a solution to Eq. \eqref{eq1} and $U_{k_0}$ is a supersolution to Eq. \eqref{eq1} according to Lemma \ref{lem.super}. We next apply a technique of \emph{separation of supports} similar to the one used in \cite{IV10}: we give a small time delay to the bigger function in order to remove the contact, and then optimize the rescaling parameter if the two functions remain uniformly separated on all their supports. More precisely, since at $t=1$ we have a perfect identification between the function and its profile, we notice that
$$
U_2(x,1)\leq U_{k_0}(x,1)<U_{k_0}(x,1+\delta),
$$
for some small $\delta>0$ fixed, since we know that our self-similar solutions advance with time both in amplitude and support. We thus find that
$$
U_2(x,1)<k_0^{-2/(m-1)}(1+\delta)^{\alpha}f_1(|x|k_0(1+\delta)^{-\beta}), \qquad |x|\in[0,\xi_2],
$$
thus the two functions are strictly separated, since $[0,\xi_2]$ is a compact set. It then follows by continuity that there exists a better parameter $k_1\in(k_0,1)$ such that we still have
$$
U_2(x,1)<k_1^{-2/(m-1)}(1+\delta)^{\alpha}f_1(k_1|x|(1+\delta)^{-\beta})=U_{k_1}(x,1+\delta).
$$
Since $U_2$ is a solution and $U_{k_1}$ is a supersolution to Eq. \eqref{eq1} according to Lemma \ref{lem.super}, we infer that $U_2(x,t)<U_{k_1}(x,t+\delta)$ for any $t>1$. Indeed, if this would not hold true, the strong maximum principle (valid for Eq. \eqref{eq1} except near $x=0$) would imply that a first contact point between $U_1(x,t)$ and $U_{k_1}(x,t+\delta)$ would appear at some time $t_0>1$ and at $x=0$. But this is not possible specifically for this kind of functions, stemming from profiles behaving near $\xi=0$ as in \eqref{beh.Q1}, as it can be again seen from Step 4 in the proof of Lemma \ref{lem.monot}. We thus get
$$
t^{\alpha}f_2(|x|t^{-\beta})\leq(t+\delta)^{\alpha}k_1^{-2/(m-1)}f_1(k_1|x|(t+\delta)^{-\beta}),
$$
or equivalently
\begin{equation}\label{interm14}
f_2(\xi)\leq\left(\frac{t+\delta}{t}\right)^{\alpha}k_1^{-2/(m-1)}f_1\left(k_1\xi\left(\frac{t}{t+\delta}\right)^{\beta}\right),
\end{equation}
for any $t>1$. We obtain by passing to the limit as $t\to\infty$ in the right hand side of \eqref{interm14} that $f_2(\xi)\leq f_{k_1}(\xi)$ for any $\xi\in[0,\xi_2]$, contradicting the choice of $k_0$ as a supremum in \eqref{opt.bar}. This contradiction gives the uniqueness of the desired profile with interface, as stated.
\end{proof}

\section{Analysis in dimension $N=1$}\label{sec.dim1}

In this final and shorter section, we deal with dimension $N=1$, where some noticeable differences appear with respect to dimensions $N\geq2$. The local analysis of the critical points in the plane is completely similar, but significant novelties appear with respect to the analysis of the critical points $Q_1$ and $Q_5$ at infinity, since the sign of $1+\sigma$ comes into play for example in \eqref{beh.Q1}. Moreover, an inspection of the system \eqref{Poincare} shows that there appears a new critical point at infinity that we call $Q_6$, which is identified in the system \eqref{systinf1} (that is, in variables $(y,z,w)$) as
$$
Q_6=\left(\frac{\sigma+2}{m-p},-\frac{(\sigma+2)(m(\sigma+1)+p)}{(m-p)^2},0\right),
$$
provided that $m(\sigma+1)+p<0$, a fact that is possible to hold true when $\sigma<-1$. Moreover, the critical point $Q_5$ lies now in the half-space $\{y>0\}$ and its local behavior changes. We gather these novelties with respect to dimensions $N\geq2$ in the following statement:
\begin{lemma}\label{lem.infty1}
\begin{enumerate}
\item In dimension $N=1$, the critical point $Q_1$ is an unstable node. The orbits going out of it contain profiles $f(\xi)$ such that $f(0)=D>0$ and if $\sigma>-1$ they can have any possible derivative $f'(0)\in\real$.
\item If $m(\sigma+1)+p\geq0$, the critical point $Q_5$ has a two-dimensional unstable manifold and a one-dimensional stable manifold. The orbits going out of it on the unstable manifold contain profiles with local behavior $f(\xi)\sim K\xi^{1/m}$ as $\xi\to0$, with $K>0$ free constant, and the unique orbit entering it lies in the infinity of the phase space.
\item If $m(\sigma+1)+p<0$, there are no orbits connecting $Q_5$ with the interior of the phase space. The critical point $Q_6$ (which exists only in this case) has a two-dimensional unstable manifold and a one-dimensional stable manifold. The orbits going out of it on the unstable manifold contain profiles with local behavior $f(\xi)\sim K\xi^{(\sigma+2)/(m-p)}$ as $\xi\to0$, with $K>0$ free constant, and the unique orbit entering it lies in the infinity of the phase space.
\end{enumerate}
\end{lemma}
\begin{proof}
The first part is immediate by inspection of the matrix $M(Q_1)$ in the proof of Lemma \ref{lem.Q1}, for $N=1$. In particular, the eigenvalue $-(N-2)$ becomes now positive. The fact that all the profiles contained in orbits going out of $Q_1$ have $f(0)=D>0$ is proved in the same way as in Lemma \ref{lem.Q1}. For $\sigma>-1$, we again get as in the proof of Lemma \ref{lem.Q1} that
$$
y\sim-\frac{z}{N+\sigma}+Cz^{1/(\sigma+2)},
$$
but now we can no longer neglect the second term and force $C=0$. Indeed, for $C=0$ we get a two-dimensional manifold containing profiles with the same behavior \eqref{beh.Q1}, while for $C\neq0$ we get
$$
y\sim Cz^{1/(\sigma+2)}, \ \ {\rm or \ equivalently} \ \ \frac{\xi f'(\xi)}{f(\xi)}\sim C\xi f(\xi)^{(p-m)/(\sigma+2)},
$$
which leads by integration to the local behavior
$$
f(\xi)\sim(D+C\xi)^{(\sigma+2)/(m-p)}, \qquad D>0, \ C\in\real \ {\rm free \ constants}.
$$
With respect to $Q_5$, we have
$$
M(Q_5)=\left(
         \begin{array}{ccc}
           -1 & -1 & 1-\frac{\beta}{m\alpha} \\
           0 & \frac{m(\sigma+1)+p}{m} & 0 \\
           0 & 0 & \frac{m+1}{m} \\
         \end{array}
       \right),
$$
hence, if $m(\sigma+1)+p>0$, the matrix $M(Q_5)$ has two positive eigenvalues and the orbits going out of it on the two-dimensional stable manifold have $y\sim 1/m$, which easily leads to the local behavior $f(\xi)\sim K\xi^{1/m}$ as $\xi\to0$. If $m(\sigma+1)+p<0$ we have a two-dimensional stable manifold and a one-dimensional unstable manifold, but by computing the eigenvectors of $M(Q_5)$ it is easy to show that the former is totally included in the invariant plane $\{w=0\}$, while the latter is included in the invariant plane $\{z=0\}$ of the system \eqref{systinf1}, thus no interesting orbits appear. Finally, if $m(\sigma+1)+p<0$, the new critical point $Q_6$ appears, and the linearization of the system \eqref{systinf1} in a neighborhood of it has the matrix
$$
M(Q_6)=\left(
         \begin{array}{ccc}
           1-\frac{2m(\sigma+2)}{m-p} & -1 & 0 \\
           \frac{(\sigma+2)(m(\sigma+1)+p)}{m-p} & 0 & 0 \\
           0 & 0 & -\frac{\sigma(m-1)+2(p-1)}{m-p} \\
         \end{array}
       \right)
$$
and it is easy to check that it has always two positive eigenvalues (one of the first two and the third). The orbits going out of it into the interior of the phase space have $y\sim(\sigma+2)/(m-p)$, which readily leads to the local behavior
\begin{equation}\label{beh.Q6}
f(\xi)\sim K\xi^{(\sigma+2)/(m-p)}, \qquad {\rm as} \ \xi\to0, \ \ K>0 \ {\rm free \ constant}.
\end{equation}
Notice that the critical points $Q_5$ and $Q_6$ coincide when exactly $m(\sigma+1)+p=0$, and that the profiles going out of them do not satisfy the contact condition $(f^m)'(0)=0$ at the interface point $\xi=0$.
\end{proof}

\noindent \textbf{Remark.} If $-2<\sigma<-1$ and $m(\sigma+1)+p<0$, there exists an explicit profile on an orbit connecting $Q_6$ to $P_0$, namely
$$
f(\xi)=K_0\xi^{(\sigma+2)/(m-p)}, \qquad K_0=\left[\frac{-(m-p)^2}{m(\sigma+2)(m(\sigma+1)+p)}\right]^{1/(m-p)},
$$
which is also a stationary solution (except at $x=0$) to Eq. \eqref{eq1}, since the time variable cancels out in \eqref{SSS}.

\medskip

\noindent In view of Lemma \ref{lem.infty1}, we split the global analysis for $N=1$ into three cases.

$\bullet$ \textbf{Case $N=1$ and $-1<\sigma<0$.} By an inspection of the proofs in the previous sections, we readily notice that Theorem \ref{th.1} remains true and with exactly the same proof, as stated.

$\bullet$ \textbf{Case $N=1$ and $-2<\sigma<-1$.} In this case, we still get the local behavior near $Q_1$ as
$$
y\sim-\frac{z}{1+\sigma}+Cz^{1/(\sigma+2)}
$$
but since $\sigma+2<1$, the first part dominates near the origin, thus \emph{all the profiles} going out of $Q_1$ behave as in \eqref{beh.Q1}, noticing that in this case they start in an increasing way near $\xi=0$ due to the change of sign of $1+\sigma$. The shooting performed in Section \ref{sec.exist} is completely similar, thus there exist orbits connecting $Q_1$ and $P_1$ and containing self-similar profiles with interface. However, the proof of the uniqueness is no longer valid (since it requires decreasing profiles) and it is in fact likely that there are infinitely many profiles in such orbits.

$\bullet$ \textbf{Case $N=1$ and $\sigma=-1$.} In this case, since $1+\sigma=0$, the local behavior near $Q_1$ is given by
$$
\frac{dy}{dz}\sim\frac{y}{z}-1,
$$
leading to a local behavior described by
$$
Y\sim-Z\ln\,\frac{Z}{X}+KZ, \qquad K>0,
$$
and the first term dominates near the origin, thus we cannot differentiate between the profiles going out of $Q_1$. Existence is again ensured by the shooting as in Section \ref{sec.exist}, while uniqueness is again not likely to be true.

We close the paper by letting open the possibility of finding a difference at the level of the second order of expansion for the profiles in these last cases.

\bigskip

\noindent \textbf{Acknowledgements} R. I. and A. S. are partially supported by the Spanish project PID2020-115273GB-I00. A. I. M. is partially supported by the Spanish project RTI2018-098743-B-100. The authors want to thank Philippe Lauren\ced{c}ot for pointing out some valuable comments and references.

\bibliographystyle{plain}

\end{document}